\documentclass[12pt]{amsart}
\usepackage[margin=1in]{geometry}
\usepackage{graphicx}
\usepackage{amsmath}
\usepackage{amssymb}
\usepackage{epstopdf}
\usepackage{subcaption}
\usepackage{amsthm}
\usepackage{xcolor}
\usepackage{mathrsfs}
\usepackage{enumitem}
\usepackage{bm}
\usepackage{bbm}

\usepackage[justification=centering]{caption}
\usepackage{makecell}

\newtheorem{theorem}{Theorem}[section]
\newtheorem*{theorem*}{Theorem}

\newtheorem{proposition}[theorem]{Proposition}
\newtheorem{corollary}[theorem]{Corollary}
\newtheorem{remark}[theorem]{Remark}
\newtheorem{assumption}[theorem]{Assumption}
\newtheorem{example}[theorem]{Example}
\newtheorem{definition}[theorem]{Definition}

\DeclareMathOperator{\argmin}{argmin}

\newcommand{\bbE}{\mathbb{E}}
\newcommand{\bbI}{\mathbb{I}}
\newcommand{\bbP}{\mathbb{P}}
\newcommand{\bbR}{\mathbb{R}}
\newcommand{\RR}{\mathbb{R}}

\newcommand{\calA}{\mathcal{A}}
\newcommand{\calC}{\mathcal{C}}
\newcommand{\calD}{\mathcal{D}}
\newcommand{\calF}{\mathcal{F}}
\newcommand{\calG}{\mathcal{G}}
\newcommand{\calL}{\mathcal{L}}

\newcommand{\calP}{\mathcal{P}}

\newcommand{\calQ}{\mathcal{Q}}
\newcommand{\calR}{\mathcal{R}}
\newcommand{\calS}{\mathcal{S}}

\newcommand{\calU}{\mathcal{U}}
\newcommand{\calV}{\mathcal{V}}

\newcommand{\scrL}{\mathscr{L}}
\newcommand{\scrM}{\mathscr{M}}
\newcommand{\scrP}{\mathscr{P}}
\newcommand{\scrS}{\mathscr{S}}

\newcommand{\Var}{\text{Var}}
\newcommand{\VaR}{\text{VaR}}

\newcommand{\peta}{p_{\eta^{\ast}}}
\newcommand{\qeta}{q_{\eta^{\ast}}}

\newcommand{\etastar}{\eta^{\ast}}

\newcommand{\nustar}{\nu^{\ast}}

\usepackage{multicol}
\usepackage[linesnumbered,lined,boxed]{algorithm2e}
\usepackage[foot]{amsaddr}
\usepackage{float}

\numberwithin{equation}{section}

\begin{document}
\title[Optimal reinsurance from an optimal transport perspective]{Optimal reinsurance from an optimal\\ transport perspective}
\date{}

\author[B. Acciaio]{Beatrice Acciaio}
\address[B. Acciaio]{Department of Mathematics, ETH Zurich, Rämistrasse 101, 8092 Zürich, Switzerland}
\email[B. Acciaio]{beatrice.acciaio@math.ethz.ch}

\author[H. Albrecher]{Hansj\"org Albrecher}
\address[H. Albrecher]{Department of Actuarial Science, Faculty of Business and Economics and Swiss Finance Institute, University of Lausanne, UNIL-Chamberonne, CH-1015 Lausanne, Switzerland}
\email[H. Albrecher]{hansjoerg.albrecher@unil.ch}

\author[B. García Flores]{Brandon García Flores}
\address[B. García Flores]{Department of Actuarial Science, Faculty of Business and Economics, University of Lausanne, UNIL-Chamberonne, CH-1015 Lausanne, Switzerland}
\email[B. García Flores]{brandon.garciaflores@unil.ch}

\begin{abstract}
{\color{black} We use the randomization idea and proof techniques from optimal transport to study 
optimal reinsurance problems. 
We start by providing conditions for a class of problems  that allow us to characterize the support of optimal treaties, and show how this can be used to deduce the shape of the optimal contract, reducing the task to an optimization problem with finitely many constraints, for which standard techniques can be applied. For a more general class of problems, we regard the optimal reinsurance problem as an iterated optimal transport problem between a (known) initial risk exposure of the insurer and an (unknown) resulting risk exposure of the reinsurer. The proposed approach provides a general framework that encompasses many reinsurance problems, which we illustrate in several concrete examples, providing alternative proofs to classical optimal reinsurance results, as well as establishing new optimality results, some of which contain optimal treaties that involve external randomness.}
\end{abstract}
\maketitle

\section{Introduction}
The identification of optimal reinsurance forms is a classical problem of actuarial risk theory. Starting with pioneering work of de Finetti \cite{definetti1940}, Borch \cite{borch1960attempt,borch1960safety} and Arrow \cite{arrow63}, with varying objective functions, constraints and choices of involved contract parties, the topic has been a rich source of interesting mathematical problems and is still a very active field of research, see for instance \cite{centeno} and \cite[Ch.8]{albrecher2017reinsurance} for an overview. Among the many conceptual and influential contributions to the topic over the last years, we mention here Kaluszka \cite{kaluszka2004mean}, Cai \& Tan \cite{cai2008optimal}, Balbas et al.\ \cite{balbas2015optimal,balbas2022risk}, and {\color{black} Cheung et al.}\ \cite{cheung2010optimal,cheung2014risk,cheung2014optimal,cheung2017characterizations} {\color{black}as well as \cite{yong2024optimal}}. For game-theoretic approaches to equilibria and efficient solutions, we refer the reader to e.g.\ \cite{asimit2018insurance,boonen2021bowley,zhu2023equilibria,boonen2023bowley}, and for a situation with several reinsurers, to Boonen \& Ghossoub \cite{boonen2021optimal}. An interesting link of optimal reinsurance problems to the Neyman-Pearson lemma of statistical hypothesis testing was established by Lo \cite{lo2017neyman}, encompassing earlier contributions such as \cite{zhuang2016marginal}; see \cite{cheung2023multi} for a recent considerable generalization of this approach. For a backward-forward optimization procedure in a rather general setting, we refer to Boonen \& Jiang \cite{boonen2022marginal}.

An intuitive practical constraint when looking for an optimal reinsurance contract is its deterministic nature, i.e., the reinsured amount is identified deterministically once the claim size of the first line insurance company is known. At the same time, one may imagine scenarios where a randomized contract (a contract where, beside the original claim size, an additional exogeneous random mechanism is used to determine the eventual reinsured amount) leads in fact to a better solution of the original optimization problem. In this case, once the realizations of the original claim sizes are available, a defined additional random mechanism is used to determine the eventual reinsured and retained amount. For instance, the proportionality factor in a proportional reinsurance contract, or the retention in an Excess-of-Loss contract, may be determined by the outcome of a random experiment with a defined distribution, which might practically be realized with the help of a lottery, or a random number simulator that all involved parties agree upon and that could be applied in the presence of a notary. While it is not straightforward to overcome the psychological barriers of such a contract formulation, if the objective and constraints of the optimization problem match the true goals of the involved parties, such a randomized solution should be preferable in case it dominates all the deterministic solutions (see \cite{albrecher2019randomized} for an extensive discussion). Gajek \& Zagrodny \cite{gajek2004reinsurance} were the first to point out such a randomized solution in a particular setup with a discrete loss distribution, where the goal was to minimize the ruin probability of the insurer for a given budget constraint. In \cite{albrecher2019randomized}, \cite{asimit2021risk} and \cite{vincent2021} it was then shown how  random reinsurance treaties can be optimal in more general situations, with problem formulations that are closer to actuarial practice. In fact, even if an insurer prefers to restrict the risk management to purely deterministic reinsurance contracts, there often is an additional random element present in any case, such as the reinsurance counterparty default risk; see e.g.\ \cite{asimit2013optimal,cai2014optimal}.

{\color{black} To express it in a simple way, the reinsurance problem can be stated as
\begin{equation}
\tag{R}\text{minimize risk}\; \mathcal{P}(\eta)\quad \text{over reinsurance treaties} \;\eta
\end{equation}  
\centerline{\text{\emph{(with $\eta$ deterministic, resp.\ random)}}}\\[0.15cm]
where $\mathcal{P}(\cdot)$ is a risk measure and $\eta$ are all feasible reinsurance contracts fulfilling the constraints of the particular problem. From a purely mathematical perspective, Guerra \& Centeno \cite{guerra2012quantile} looked for optimal reinsurance contracts in a specific situation, found solutions of a potentially randomized form for their particular one-dimensional model setup, and showed that there always exists an optimal nonrandom treaty, so that their randomization feature could be interpreted merely as a mathematical tool.  
In a certain way, part of the subsequent analysis can be seen as a generalization of their approach to arbitrary risk measures in a multi-dimensional setting, where additional constraints on the solutions are allowed, and where deterministic treaties are not necessarily optimal anymore. 

In the present paper, we take a new perspective on the reinsurance problem, that allows us to use tools and ideas from the optimal transport (OT) theory. 
Roughly speaking, the latter is concerned with transporting mass from a source distribution $\mu$, to a target distribution $\nu$, while minimizing the total transport according to a cost function $c$:
\begin{equation}
\tag{T}\text{minimize transport cost}\;\textstyle{\int} c\ d\pi\quad\text{over couplings}\;\pi\; \text{of the marginals} \;\mu,\nu
\end{equation}  
\centerline{\text{\emph{(with $\pi$ supported on the graph of a map, resp.\ not).}}}\\[0.15cm]
The OT theory is an extremely active field of research, that has played a key role in many areas of applied mathematics ranging from PDEs, image processing, inverse problems, sampling, optimization, finance and economics to machine learning. We refer to the monographs by Villani~\cite{villani2021topics},  Santambrogio~\cite{santambrogio2015optimal} and Ambrosio et al.~\cite{ambrosio2021lectures} for an overview of the theory, as well as to \cite{peyre2019computational} for the computational development, and to \cite{galichon2018optimal,henry2017model} for applications in economics and finance.

Here we look at the reinsurance contract as a mechanism to reshape the loss distribution $\mu$ of the first-line insurer into a less risky one. The reinsurer enables that improvement by participating with a loss distribution $\nu$ in the coverage of the claim, and will ask for a reinsurance premium for this service. 
This point of view prompts the analogy between problem (R) and problem (T). Here, the treaty $\eta$ takes the role of the coupling $\pi$ between the initial loss distribution $\mu$ of the insurer and the loss distribution $\nu$ of the reinsurer.
The criteria chosen to set up this optimization problem define the cost faced by the insurer to transfer the risk. 
When restricting our attention to deterministic reinsurance contracts, we confine the OT problem to particular types of couplings, supported on the graph of a map, known as Monge transports \cite{monge1781memoire}. This means that the final loss of the insurer (and correspondingly also that of the  reinsurer) is simply a deterministic function of the original loss.  On the other hand, allowing for randomized reinsurance treaties corresponds to allowing general couplings, as in the general so-called \textit{Kantorovich formulation} of the OT problem \cite{Kantorovich}. In this case, the original loss size alone does not necessarily already determine the reinsured amount. 

We need to point out some obstacles on the way of our goal to connect the optimal reinsurance to the optimal transport problem.  First of all, while in OT the two marginals are known, in the reinsurance setting typically only the first marginal $\mu$, that is the loss distribution of the insurer (before reinsurance) is initially known, while the target distribution $\nu$, that is the loss distribution of the reinsurer, needs to be found as a result of the optimization (typically, one looks for the best improvement of the risk measure $\mathcal{P}$ for a given budget constraint). This means that we have an additional layer of complexity to consider in our problem, which in some cases can be expressed as a further optimization over all possible target distributions.
Another fundamental difference is that, while in OT problems the cost functional is linear in the measure, in optimal reinsurance problems -- even for simple risk measures such as the variance -- linearity does not hold. Consequently, in order to establish a link and apply OT techniques, we need to linearize the latter. The easiest way to that end is by taking derivatives, but already at this point one runs into challenges since there are several notions of differentiability in infinite dimensions (e.g., the Fréchet or Gateaux type; and even for the latter one might need to impose linearity, continuity, etc.). Notice that one cannot ask for a very strong notion of differentiability, since this implies continuity on opens sets (in the space of signed measures), which is hardly satisfied for situations of practical interest (e.g., not even for the expectation operator). For our purposes, it will turn out to be sufficient (and lead to interesting results already) that the directional derivatives exist and that they are convex linear.\footnote{\color{black}By \textit{convex linear} we mean that the linearity property is only satisfied in convex combinations of the arguments, i.e. $T$ is convex linear if $T(sx+(1-s)y)=sT(x)+(1-s)T(y)$, $0\leq s\leq 1$.}
If we suppose that these derivatives are integral operators, then, for simple constraints, it turns out that we can often characterize the support of optimal treaties up to a finite set of parameters; see Section~\ref{sec3} for details. In the case of more general constraints, it will be possible to recast the problem as an iterated OT problem, which will allow some further insight and deductions; see Section~\ref{secOT} for details.\\
}

{\color{black}\noindent The main results of the paper are Propositions \ref{prop:support_minimizer} and \ref{prop:suppport_minimum}, as well as the discussion in Section \ref{secOT}, which provide the principal tools to explicitly solve reinsurance problems in Section \ref{sec5}. In addition, on a more heuristic level, we believe that a main contribution of our work is the invitation to rethink the classical setting in optimal reinsurance from an optimal transport perspective. In this sense, Proposition \ref {remark:monge_1} and Remark \ref{prop:monotonic_rea} show that, in dimension 1, the condition that reinsurance amounts can not exceed the original claim size (concretely, Condition \eqref{thatt} in the next section) is equivalent to optimization over measures stochastically dominated by the distribution of the insurer's risk. That is,  classical conditions for reinsurance contracts can be rephrased into well-studied concepts of distributions. Example~\ref{ex:ex6} then represents a concrete case where adopting the approach proposed in this paper leads to significant new results.\\}

{\color{black}
\noindent{\bf Notations.} 
For $p\in\mathbb{N}$ and a measurable set $A\subseteq \bbR^{p}$, we denote by $\scrP(A)$ the set of probability measures on the Borel sets of $A$. For $\xi\in\scrP(\bbR^{p})$ and $X\sim\xi$, with an abuse of notation, we use both $F_{\xi}$ and $F_X$ for the cumulative distribution function. We denote by $\bbR^{p}_{+}$ the subset of $\bbR^{p}$ where all the coordinates are non-negative.
For $\mu,\nu\in\scrP(\bbR^p)$, we write $\nu\prec_1\mu$ if $F_\nu(x)\geq F_\mu(x)$
for every $x\in \bbR^p$. For $p,q\in\mathbb{N}$, the push-forward measure of $\xi \in \scrP (\RR^p)$ through a measurable map $f: \RR^p \rightarrow \RR^q$, is the probability measure $f_\#\xi\in \scrP (\RR^q)$ such that $f_\#\xi(A)=\xi(f^{-1}(A))$ for any Borel set $A$ in $\RR^q$.  
If $V=\bbR^{n}$ and $v\in V$, we understand the inequalities $f(u)\leq v$, $f(u)<v$, etc. as holding component-wise (with this convention, $a\not\leq b$ does not imply $b<a$ unless $n=1$).\\

\noindent{\bf Organization of the paper.} The rest of the paper is organized as follows.
In Section~\ref{sect.connect} we recall the reinsurance problem and the optimal transport problem, pointing out the connections between the two, as well as the challenges one faces when trying to rephrase the former via the latter. 
In Section~\ref{sec:random_ri} we define the general optimal reinsurance problem considered in this paper. Section~\ref{sec3} then considers in more detail the situation with finitely many constraints. Here we show how this situation can be cast into a setup analogous to Lagrange optimization of a multivariate function under constraints. In Section~\ref{secOT}  we consider a more general set of constraints, and recast} the optimal reinsurance problem into the framework of an (iterated) optimal transport problem, which involves a local linearization of the optimization problem. The established link of the two fields then allows a characterization of reinsurance problems with purely deterministic optimal treaties. In Section~\ref{sec5} we then apply the results to a number of examples, some of which rederive classical optimality results by alternative means, while others lead to new results, and we compare these to existing literature. Some technical derivations in these examples are deferred to the appendix. Finally, Section~\ref{secconc} concludes. 
\section{{\color{black}The link between optimal reinsurance and optimal transport}}\label{sect.connect}
{\color{black}
This section is devoted to give the core idea of the link between reinsurance problems and optimal transport. The precise respective relations and statements will then be worked out in the later sections.\\

\noindent{\bf The reinsurance problem.}
Consider a fixed probability space $(\Omega,\mathscr{F},\bbP)$, and a non-negative random vector $X=(X_{1},\ldots,X_{n})$ defined on it. $X$ represents a portfolio of $n$ risks of one or more insurers, which are sought to be partially reinsured. We denote the distribution of each $X_i$ by $\mu_i$, the joint distribution of all $X_i$'s by $\mu$, and we assume that each $X_{i}$ has a finite first moment. In the classical formulation of optimal reinsurance problems, $X$ is defined in some function space $\scrL$ (normally an $L^{p}$-space) and the objective is to find the best contract according to a pre-specified \textit{risk measure}, i.e., a functional $\calP:\scrL\to \bbR$ 
representing the risk remaining after the treaty. A reinsurance contract is then given by a sequence of functions $f_i:\bbR_{+}^n\to \bbR_{+}$ such that 
\begin{equation}\label{thatt}
    0\leq f_i(x_1,\ldots,x_n)\leq x_i\quad \text{for}\;(x_1,\ldots,x_n)\in \bbR_{+}^n\;\text{and}\;i=1,\ldots,n, 
\end{equation}
where $f_i(X)$ is the reinsured amount, and $X_i-f_i(X)$ is the amount regarding risk $X_i$ that is retained by the insurer. Condition \eqref{thatt} ensures that these quantities are non-negative\footnote{Note that in this general formulation, the reinsured amount in contract $i$ is allowed to depend on the claim sizes of the other contracts as well, which in classical contracts is usually not the case, but we are interested here to explore whether that can lead to improved solutions.}. The optimal reinsurance problem is usually accompanied by a set of constraints (beliefs, conditions or requirements) that either of the involved parties might require.

For example, if the reinsurance premiums are calculated according to an expectation principle, the first-line insurer might ask for $\bbE[f_i(X)]\leq c_i$ for some $c_i> 0$, ensuring that reinsurance is not too expensive. All such constraints can be represented by a set $\calS$ of tuples of functions, in the sense that a contract $(f_1,\ldots, f_n)$ is allowed if and only if 
$
(f_1,\ldots,f_n)\in \calS.
$
In this example, $\calS$ could then simply be 
\[
\calS = \{(f_1,\ldots,f_n) \mid (f_1(X),\ldots,f_n(X))\in \scrL, \bbE[f_i(X)]\leq c_i, i=1,\ldots, n\}.
\]
The problem formulation then is: find a sequence of functions $f_1^{\ast}, \ldots, f_n^{\ast}$ such that
\begin{equation}\label{eq:reins}
\calP(f_1^{\ast}(X), \ldots, f_n^{\ast}(X)) = \min\limits_{\substack{0\leq f_i(x)\leq {x_i}\\ (f_1,\ldots,f_n)\in \calS}} \calP(f_1(X), \ldots, f_n(X)).
\end{equation}
\vspace{0.1cm}

\noindent{\bf Deterministic reinsurance treaties and the Monge problem.}
The idea of this manuscript is to leverage ideas from the theory of optimal transport to solve reinsurance optimization problems like the one stated above. This link looks natural, as in its foundations, optimal transport served as the study of optimal mass transfer, where the role of mass in the reinsurance setting is the ``risk'' expressed through the density function (or probability mass function) of the insurer, and reinsurance reshapes that risk into a new (safer) density function. 

In the original transport formulation by Monge \cite{monge1781memoire},  given a `source' distribution $\mu$ in $\bbR^n$, a `target' distribution $\nu$ in $\bbR^n$, and a cost function $c:\bbR^n\times\bbR^n\to[0,\infty)$ (where $c(x,y)$ expresses how costly is to `transport' a unit of mass from a position $x$ to a new position $y$), one looks for functions $g:\bbR^n\to \bbR^n$ that map (transport) $\mu$ into $\nu$, in the sense that $\nu(A) = \mu(g^{-1}(A))$ for all Borel measurable sets $A\subset \bbR^n$.\footnote{We keep the same dimension $n$ for both $\mu$ and $\nu$ as this is what matters in the present paper, although the OT problem can be dealt with in more generality.}
The \textit{Monge formulation} of the OT problem is then to find the map $g^*$ from $\mu$ to $\nu$ that minimizes the aggregate transport cost, so that
\begin{equation}\label{eq:motivation_2}
    \int_{\bbR^n} c(x,g^{\ast}(x))\mu(dx) = \inf_{\substack{\nu(A) = \mu(g^{-1}(A))\\ \forall A \text{ Borel}}} \int_{\bbR^n} c(x,g(x))\mu(dx).
\end{equation}
By collecting the functions $f_i^\ast$ appearing in the reinsurance setting into a single function $f^\ast:\bbR^n\to\bbR^n$, one can notice a resemblance between problem \eqref{eq:reins} and \eqref{eq:motivation_2}. In particular, the analogue of the cost in the OT formulation is the risk measure (and not the reinsurance premium) to be minimized. The reinsurance premium enters as a constraint on feasible contracts (e.g.\ through the definition of the set $\calS$ as above).


\begin{example}\label{ex:ex1intro}\normalfont
To make the analogy more explicit with a concrete example, consider $n=1$ and assume that the insurance risk $X$ with distribution $\mu$ has finite variance. Let $\calP$ and $\calS$ be given by 
\begin{equation}\label{introvar}
	\calP(f)= \text{Var}(X-f(X))=\int (x-f(x))^{2}\,\mu(dx) - \left(\int (x-f(x)) \, \mu(dx)\right)^{2}
\end{equation}
and $\calS=\{f\mid \int f(x)\mu(dx) = c\}$ for some $c\geq 0$. This is the classical example (see e.g.\ \cite{Pesonen1984}) where one looks for a deterministic reinsurance of the form $f(X)$ and the objective is to minimize the retained variance of the insurer subject to a fixed reinsurance premium which is computed through the expected value principle. The optimal reinsurance problem can then be written as 
\begin{equation}\label{prec10}
\begin{split}
\inf_{f\in \calS}\calP(f)&= \inf_{f\in S}\int_{\bbR\times\bbR}(x-f(x))^2\mu(dx)-\left(\int_{\bbR\times\bbR}(x-f(x))\mu(dx)\right)^2 \\
&= \inf_{f\in \calS}\int_{\bbR\times\bbR}(x-f(x))^2\mu(dx)-\left({\mathbb E}(X)-c\right)^2.
\end{split}
\end{equation}
Now, let $\calS'$ denote the set of distributions in $\scrP(\bbR)$ with expectation $c$ which are push-forwards of $\mu$ by reinsurance contracts. Then, a distribution $\nu$ with mean $c$ belongs to $\calS'$ if and only if $\nu=f_\#\mu$ for some reinsurance contract $f$. The infimum in the last line of \eqref{prec10} can be written as
\begin{equation}\label{prec10_2}
\inf_{f\in \calS}\int_{\bbR\times\bbR}(x-f(x))^2\mu(dx) = \inf_{\nu \in \calS'}\inf_{\substack{\nu(A) = \mu(f^{-1}(A))\\ \forall A \;\text{Borel}}}\int_{\bbR\times\bbR}(x-f(x))^2\mu(dx).
\end{equation}
Notice that, for fixed $\nu\in \calS'$, the inner infimum on the right-hand side of \eqref{prec10_2} is an OT problem of the form \eqref{eq:motivation_2} for the cost function $c(x,y)=(x-y)^2$, so solving the OT problem will provide information about the solution to the stated optimal reinsurance problem. 

We recognize that while \eqref{prec10_2} helps connect the theory of OT to optimal reinsurance, this might in principle over-complicate things, considering the complex description of the set $\calS'$. Part of our results shows that this set can be replaced by a more manageable set that makes no reference to reinsurance contracts. We will return to the solution of the example later in the paper, in Example \ref{ex:ex1}.\hfill $\diamond$ 
\end{example}
\vspace{0.1cm}

\noindent{\bf Random reinsurance treaties and the Kantorovich problem.}
So far in this section, we restricted the considerations to deterministic reinsurance forms (deterministic functions $f$). However, the OT approach allows to be much more general. Instead of insisting on maps that move $\mu$ into $\nu$, one may allow `mass splitting', meaning that the mass sitting on a point $x$ needs not be transported into a unique place $y$. Mathematically this translates into searching for joint distributions $\pi$ on $\bbR^n\times\bbR^n$ with marginals $\mu$ and $\nu$ (called 'couplings' of $\mu$ and $\nu$). This means that the mass sitting on $x$ is now moved along the disintegration $\pi^x$ of $\pi$ w.r.t. $\mu$.\footnote{{\color{black}With disintegration of $\pi$ w.r.t. $\mu$ we mean that $\pi(dx,dy)=\mu(dx)\pi^x(dy)$.}} The \textit{Kantorovich formulation} of the OT problem then reads as
\begin{equation}\label{eq.ot_intro}
\inf_{\pi\in\Pi(\mu,\nu)}\int_{\bbR^n\times\bbR^n}c(x,y)\pi(dx,dy),
\end{equation}
where $\Pi(\mu,\nu)$  is the set of couplings (or transport plans) between $\mu$ and $\nu$,
\begin{equation}\label{eq.pimunu}
\Pi(\mu,\nu)=\{\eta\in\scrP(\bbR^{n}\times\bbR^{n})\mid {\pi_1}_\#\eta=\mu, {\pi_2}_\#\eta=\nu\},
\end{equation}
with
$\pi_{i}:\bbR^{n}\times\bbR^{n}\to\bbR^{n}$, $i=1,2$  projections into first and second marginal, respectively (see \cite{villani2021topics,ambrosio2021lectures}).
\footnote{{\color{black}Recall that the push-forward measure ${\pi_i}_\#\eta$ is the probability measure on $\scrP (\RR^n)$ such that ${\pi_i}_\#\eta(A)=\eta({\pi_i}^{-1}(A))$, for any Borel set $A$ in $\RR^n$.}} 
The success of formulation \eqref{eq.ot_intro} is due to the fact that, in contrast to \eqref{eq:motivation_2}, it is a linear problem and always well-posed (the set of couplings is always non-empty, as it contains the independent coupling $\pi(dx,dy)=\mu(dx)\nu(dy)$).
An optimizer $\pi^*$ of \eqref{eq.ot_intro} is called optimal coupling. If it is concentrated on the support of a map $g:\bbR^{n}\to\bbR^{n}$, i.e.\ $\pi^*(dx,dy)=\mu(dx)\delta_{g(x)}(dy)=(\mathrm{Id},g)_\#\mu$ (where $\mathrm{Id}$ is the identity function), then it is called an optimal Monge coupling, and $g$ is an optimal Monge map.

In reinsurance terms, allowing for this mass split corresponds to admitting \emph{random reinsurance treaties}, i.e.\ treaties for which the reinsured (and the retained) loss is not simply function of the original insurance loss, but may involve some additional random component, reflected in the coupling $\eta$.
In the setting of Example \ref{ex:ex1intro}, this means formulating \eqref{prec10} in terms of random reinsurance treaties. The effect of this is simply a rewriting of the inner infimum on the right-hand side of \eqref{prec10_2} in terms of the Kantorovich formulation \eqref{eq.ot_intro}. If the solution to the inner infimum always leads to a non-deterministic coupling between $\mu$ and $\nu$  (i.e., a coupling that does not have support on the graph of a function), then that carries over to the optimal reinsurance form as well. We will see later in the paper that for Example \ref{ex:ex1intro} the deterministic solution cannot be outperformed by any randomized one. On the other hand, we will identify other reinsurance problems for which randomized solutions are indeed preferable. This shows how the OT approach is a natural tool to study the optimality of random reinsurance treaties. In the next sections we will introduce this concept more formally.\\

\noindent{\bf Challenges in the connection between optimal reinsurance and optimal transport.}
We summarize in Table~\ref{tab:ot_and_reinsurance} below the resemblance between the optimal reinsurance and the optimal transport problems, which can be thought of as a guide on how the ingredients from OT map into the reinsurance setting.

\begin{table}[h]
{\color{black}
    \centering
    \begin{tabular}{|c|c|}
    \hline
        \thead{OT} & \thead{Reinsurance}\\
        \hline
        Pre-fixed target measure $\nu$& \makecell{Constraints on the contract characteristics, \\ but $\nu$ not yet specified}\\
        \hline
        Cost (integral over $c$) & Risk measure (e.g. variance of the remaining risk)\\
        \hline
        Monge maps $g$ & Classical reinsurance contract $f$\\
        \hline
        Kantorovic coupling $\pi$ & Random reinsurance treaty $\eta$\\
        \hline
    \end{tabular}
    \caption{{\color{black}Analogy between the OT problem and the reinsurance problem}}
\label{tab:ot_and_reinsurance}
}
\end{table}
By looking at Table~\ref{tab:ot_and_reinsurance}, it is clear that not every choice of risk measure $\calP$ and constraints set $\calS$ will result in a proper OT problem. Even more, there are some fundamental aspects that differ in the reinsurance and transport problems.

\emph{(i)} First, the target measure $\nu$  is fixed (pre-given) in the OT setting, while the optimal reinsurance problems typically do not specify the target (reinsurer's loss) distribution $\nu$, but aim to find the best way to minimize the risk measure $\calP$ for the given budget (reinsurance premium). Therefore, when we want to utilize OT techniques for the solution of reinsurance problems, we will have to introduce two steps: first, finding the best mapping for a given $\nu$, and then, optimizing over all feasible choices of $\nu$. 

\emph{(ii)} Further, the total cost appearing in \eqref{eq.ot_intro} is linear in $\pi$. As linearity of the risk measures would be a rather restrictive property to impose, we instead explore the idea of (locally) linearizing the reinsurance problem, a technique often employed within the context of finite-dimensional optimization. 

\emph{(iii)} Another important factor in the correspondence of reinsurance and transport problems is that the latter is posed in terms of distributions. In order to transfer this property to the reinsurance setting, we will assume throughout the paper that the risk measure $\calP$ is law-invariant, which is commonly required in the optimal reinsurance literature and will allow us to phrase the optimization problems in terms of the distribution of $X$. Notice that we are also required to impose that the set of constraints $\calS$ can be described by means of the distribution of $X$.

Given the above comments, it is clear that one cannot simply apply tools from transport theory in order to solve reinsurance problems. The point of the manuscript is, however, not to provide a one-to-one correspondence between optimal reinsurance problems and OT problems, but instead to examine which ideas from OT translate into the reinsurance setting. In the subsequent sections, we delve into specific assumptions about $\calP$ or $\calS$ that allow us to reach more concrete conclusions about optimal contracts, while trying to preserve a level of generality suitable for a wide range of applications.
For example, one crucial step in OT is the identification of the support of optimal couplings. Following this intuition, one can see how most of the results obtained in Section~\ref{sec3} are based on the idea of identifying the support of optimal couplings, which in turn serves as base for the examples developed in Section~\ref{sec5}.

}

\section{Problem formulation}\label{sec:random_ri}

{\color{black}We introduce reinsurance treaties as a multi-dimensional generalization of the concept used in \cite{guerra2012quantile} for the one-dimensional case (see also \cite{guerra2021reinsurance} for another direction in the multi-dimensional case). 
\begin{definition}\label{def:ReInTreaty}
A (random) reinsurance treaty for the portfolio $X\sim\mu\in \scrP(\bbR^n_+)$ is a probability measure $\eta\in \scrP(\bbR^n_+\times \bbR^{n}_+)$ such that:
\begin{itemize}
\item[(i)] the first marginal of $\eta$ equals $\mu$, i.e. ${\pi_1}_\#\eta=\mu$;
\item[(ii)] $\eta(\mathcal{A}_R)=1$, 
where $\mathcal{A}_R=
\{(x,y)\in \bbR^n\times \bbR^{n}\mid 0\leq y_{i}\leq x_{i}, i=1,\ldots, n\}$.
\end{itemize}
We denote the space of reinsurance treaties as $\scrM$ and endow it with the weak topology.\footnote{{\color{black}We drop dependence on $\mu$ to lighten the notation, as the initial loss distribution $\mu$ of the first-line insurer is considered as fixed throughout
the paper}}
\end{definition}}
{\color{black}We will henceforth assume that $\mu$ is absolutely continuous w.r.t. the Lebesgue measure.}

Random treaties as in the definition above should be understood as reinsurance contracts in which the risk carried by the cedent (and the reinsurer) possibly has a degree of randomness external to the risks represented in $X$. Indeed, for any reinsurance treaty for the portfolio $X\in\bbR^{n}_+$, there exists a random vector $R\in\bbR^{n}_+$ representing the part of the risks carried by the reinsurer, while $X-R\in\bbR^{n}_+$ is the retained amount (deductible) that stays with the first-line insurer(s). Then $\eta$ in Definition~\ref{def:ReInTreaty} is the distribution of the random vector $(X,R)$. Observe that, much as with the functions $f_1,\ldots, f_n$ in the previous section, reinsurance treaties also specify how contracts are settled: given a realization of the claims $X=x$, one uses $\eta$ to determine the conditional distribution of $R$ given $X=x$. One then uses this conditional distribution to sample a value for $R$, say $r$, thus obtaining the retained amount $r$ and the deductible $x-r$. The fact that $\eta$ is supported on $\calA_R$ thus guarantees that both $r$ and $x-r$ are non-negative.
{\color{black}The classical (non-random) treaties are clearly a particular case, where $\eta$ has its support on the graph of a function, say $g:\bbR^n\to\bbR^n$, so that, given a realization $X=x$ of the claims, the reinsured part is given by $r=g(x)$.}

\begin{remark}\normalfont
The weak topology appearing in Definition~\ref{def:ReInTreaty} corresponds to the smallest topology making the functionals $\eta\mapsto \int f d\eta$ continuous, where $f$ is any bounded and continuous function on $\bbR^{n}\times\bbR^{n}$.
This definition can be seen as the multi-dimensional generalization of the one in \cite{guerra2012quantile}, but different from the one found in \cite{guerra2021reinsurance}. Compared to the latter, we allow for dependencies between the reinsured amounts and all the risks in the portfolio, which makes $\scrM$ a compact space, instead of relatively compact.  \hfill $\diamond$ 
\end{remark}

{\color{black}With the above definition, we can extend the remaining elements of the optimal reinsurance problem to fit into this framework. Hence, the risk measure\footnote{{\color{black} Notice that we use the term ``risk measure'' in a rather lax manner to simply stand for a function mapping random reinsurance contracts to the real numbers. 
}} is now seen as a functional $\calP:\scrM\to \bbR$ 
and the set of constraints $\calS$ as a subset of $\scrM$. Notice that since $\scrM$ depends on $X$ only through its distribution, $\calP$ and $\calS$ defined in this way are law-invariant.} The optimal reinsurance problem consists then in looking for $\eta^{*}\in \calS$ such that
\begin{equation}\label{eq.prob}
\calP(\eta^{*}) = \min\limits_{\eta\in \calS} \calP(\eta).
\end{equation}
Any $\eta^*$ satisfying \eqref{eq.prob} is called an \emph{optimal reinsurance contract}.

Under rather mild assumptions on the functional $\calP$ and the set of constraints $\calS$, the existence of optimal treaties is guaranteed.
\begin{proposition}\label{prop:existence}
If $\calP$ is lower semi-continuous and $\calS$ is closed, then an optimal treaty $\eta^{*}$ exists.
\end{proposition}
\begin{proof}
The proof carries over from the proof of from Proposition 1 in \cite{guerra2012quantile} showing that $\scrM$ is compact (observe that the proof of that proposition shows only that $\scrM$ is sequentially compact; however, since $\scrM$ is metrizable, the two notions agree). Since $\calS$ is closed, it is compact as well, ensuring that $\calP$ attains a minimum.
\end{proof}
Throughout the rest of the paper we will assume that $\calP$ is lower semi-continuous and that $\calS$ is closed.

\section{\color{black}Optimal reinsurance with finitely many constraints}\label{sec3}
As mentioned above, we want to use the idea of linearization to deduce further properties of optimal reinsurance treaties. In order to do this, it is natural to impose smoothness properties on $\calP$. {\color{black}In this section, we develop these ideas
when $\calS$ has the particular form
\begin{equation}\label{eq:shapeofS}
    \calS=\{\eta\in\scrM \mid \calG(\eta)\leq 0\}
\end{equation}
for a lower semi-continuous function $\calG=(g_1,\ldots,g_m):\scrM\to \bbR^{m}$. That is, we allow for $m$ constraints expressed as inequalities $g_i(\eta)\leq 0$, $i=1,\ldots,m$.

We first need to introduce the notion of directional derivative.} 
For a function $f:\Omega_c\to V$ from a convex subset $\Omega_c$ of a vector space $U$ into a normed space $V$, we let $df(u;h)$ denote the directional derivative of $f$ at $u\in \Omega_c$ in the direction of $h\in U$, given by 
\[
df(u;h) = \lim\limits_{t\to 0^{+}}\frac{f(u+th)-f(u)}{t},
\]
whenever the terms on the right are defined and the limit exists. Notice that this derivative is defined as the right-limit at zero,  and by stating that $df(u;h)$ exists we make no assumption about the existence or value of the left limit (thus distinguishing $df$ from the Gateaux derivative). 

\begin{remark}\normalfont Observe that when speaking of the directional derivatives of $\calP$, it is meaningless to try to obtain expressions of the form $d\calP(\eta;\vartheta)$ for $\eta,\vartheta\in \scrM$ since for any $t>0$, $\eta +t\vartheta$ will not be a probability measure, so the expression $\calP(\eta +t\vartheta)$ in the definition of $d\calP(\eta;\vartheta)$ would not make sense. However, by exploiting the convexity of $\scrM$, we observe that for $0<t<1$, $\eta+t(\vartheta-\eta)= (1-t)\eta+t\vartheta\in \scrM$ and it is meaningful to inquire about the existence and properties of $d\calP(\eta;\vartheta-\eta)$, which we will do in the sequel. These directional derivatives have the simple and natural interpretation of the instant change in $\calP$ when, standing in $\eta$, we move in the direction of $\vartheta$. In this context, the use of directional derivatives resembles the notion of Gateaux-differentiability employed by Deprez and Gerber in \cite{deprez1985convex} for solving problems within optimal reinsurance and optimal cooperation. In their context, there is a risk functional $\widehat{\calP}$ acting on random variables (instead of their distributions) and one obtains the derivatives by limits of the form $(\widehat{\calP}(Y+t(Z-Y))-\widehat{\calP}(Y))/t$. This is, however, different from our approach, because when $Y$ and $Z$ are distributed according to $\eta$ and $\vartheta$ respectively, $(1-t)Y+tZ$ will in general not be distributed according to $(1-t)\eta+t\vartheta$. \hfill $\diamond$ 
\end{remark}
Before proceeding, we would like to emphasize that the following approach in particular applies to constraints involving the Value-at-Risk (which, compared to other common risk measures, exhibits poor mathematical properties), for which other approaches often fail. 
\begin{proposition}\label{prop:kkt}
Let $\calS$ be given by the set $\{\eta\in\scrM \mid \calG(\eta)\leq 0\}$ for a function $\calG:\scrM\to \bbR^{m}$, and $\etastar\in \calS$ be an optimal reinsurance contract, i.e. satisfying \eqref{eq.prob}. Moreover, set
\[
\calD = \{\eta\in\scrM \mid d\calP(\etastar;\eta-\etastar) \text{ and } d\calG(\etastar;\eta-\etastar) \text{ exist}\}.
\]
Suppose there exists a subset $\calC$ of $\calD$ satisfying: $\etastar\in \calC$,
$\calC$ is convex, and if $\eta_{1},\eta_{2}\in \calC$ then
\begin{equation}\label{eq:convex_gateaux}
d\calP(\etastar;(1-t)\eta_{1}+t\eta_{2}-\etastar) = (1-t)d\calP(\etastar;\eta_{1}-\etastar)+td\calP(\etastar;\eta_{2}-\etastar)\quad \text{for $0\leq t\leq 1$},
\end{equation}
and similarly for $d\calG$. Then, there exist $r^{\ast}\in\bbR_+$ and $\lambda^{\ast}\in \bbR^{m}_+$ such that $\lambda^{\ast}\cdot \calG(\etastar)=0$ and
\begin{equation}\label{eq:positive_lagrange}
r^{\ast}d\calP(\etastar;\eta-\etastar)+\lambda^{\ast}\cdot d\calG(\etastar;\eta-\etastar) \geq 0\quad \text{for every $\eta\in \calC$}.
\end{equation}
If $\calG$ is constant on $\calC$ or there exists $\eta\in \calC$ such that $\calG(\etastar)+d\calG(\etastar;\eta-\etastar)<0$, then $r^{\ast}$ in \eqref{eq:positive_lagrange} is positive.
\end{proposition}
\begin{remark}\normalfont\label{remthis}
Before showing a proof for the proposition, we would like to provide an intuitive explanation of its meaning. Recall the standard optimization procedure for differentiable functions on $\bbR^{N}$ with smooth constraints, where one forms the Lagrangian and solves for its gradient while finding the multipliers that make the solutions satisfy the constraints. Proposition~\ref{prop:kkt} generalizes this procedure and shows that, in the current scenario, one can still do something similar (observe that the left-hand side of \eqref{eq:positive_lagrange} would correspond to the derivative of the Lagrangian of the associated functionals). However, due to the looseness of our assumptions, the conclusion is much weaker. Indeed, \eqref{eq:positive_lagrange} is only an inequality instead of an equality. This is due to the fact that $\calP$ and $\calG$ are, in principle, defined on $\scrM$ only rather than on the much larger space $\scrS$  of finite signed measures. Next, observe that we have to resort to the use of directional derivatives, as opposed to a stronger concept of differentiability. Using directional derivatives we cannot ensure linearity in the second argument of $d\calP$ or $d\calG$, so that we are forced to operate on a smaller subset of $\scrM$. It is here that the set $\calC$ plays a relevant role. One can think of this set as a ``large enough set in which both $\calP$ and $\calG$ are smooth'' (see also Remark~\ref{rem:size_C} on how ``large enough'' could be understood). Since, for arbitrary functionals, $\calC$ will be a strict subset of $\scrM$, the information provided by \eqref{eq:positive_lagrange} will not hold for all possible reinsurance contracts. Finally, notice that, unlike the usual Lagrangian in finite dimensions, there is a factor $r^{\ast}$ multiplying $\calP$. The appearance of this factor is related to the regularity constraint $\calG(\etastar)+d\calG(\etastar;\mu-\etastar)<0$, which essentially eliminates the possibility of a wedge at boundary points. {\color{black}In particular, this also prevents us from imposing \textit{hard} equality constraints, which can be obtained with inequalities by replacing $\calG$ with $(\calG, -\calG)$.} 
Despite these shortcomings, we will see that Proposition~\ref{prop:kkt} is still strong enough to derive important properties of $\etastar$. \hfill $\diamond$ 
\end{remark}
\begin{proof}[Proof of Proposition~\ref{prop:kkt}]
Define the sets 
\begin{align*}
A &= \{(r,\lambda) \in \bbR\times \bbR^{m}\mid r\geq d\calP(\etastar;\eta-\etastar), \lambda\geq \calG(\etastar)+d\calG(\etastar;\eta-\etastar) \text{ for some } \eta\in \calC\},\\
B &= \{(r,\lambda) \in \bbR\times \bbR^{m}\mid r< 0, \lambda< 0\}.
\end{align*}
Observe that $A$ is non-empty since $\calC$ is non-empty, and convex thanks to  \eqref{eq:convex_gateaux}. On the other hand, $B$ is clearly convex, non-empty and open. We claim that $A$ and $B$ are disjoint. Arguing by contradiction, suppose there is some $r<0$, $\lambda<0$ and $\eta \in \calC$ such that 
\[
d\calP(\etastar;\eta-\etastar) \leq r\quad \text{and}\quad \quad \calG(\etastar)+d\calG(\etastar;\eta-\etastar)\leq \lambda.
\]
Notice that the first inequality implies $\eta\neq\etastar$. Since
\[
\lim\limits_{t\to 0^{+}} \calP(\etastar + t(\eta-\etastar)) - \calP(\etastar)- td\calP(\etastar;\eta-\etastar) = 0
\]
and 
\[
\lim\limits_{t\to 0^{+}} \calG(\etastar + t(\eta-\etastar)) - \calG(\etastar)- td\calG(\etastar;\eta-\etastar) = 0,
\]
it follows that there exists $0<s<1$ such that 
\[
\calG(\etastar + s(\eta-\etastar))< 0 \text{ and } \calP(\etastar + s(\eta-\etastar)) - \calP(\etastar)<0,
\]
so that $\etastar + s(\eta-\etastar)\in \calS$ and $\calP(\etastar + s(\eta-\etastar)) < \calP(\etastar)$, contradicting the optimality of $\etastar$. Therefore, $A\cap B=\varnothing$. From the separation theorem for convex sets (see e.g. Theorem 3.4 in \cite{Rudin1991}), there exists a (non-zero) continuous linear functional $\Lambda^{\ast}$ on $\bbR\times\bbR^{m}$ and $\gamma\in \bbR$ such that 
\[
\Lambda^{\ast}(r_{1},\lambda_{1}) <\gamma \leq \Lambda^{\ast}(r_{2},\lambda_{2}), \quad (r_{1},\lambda_{1})\in B,(r_{2},\lambda_{2}) \in A
\]
and $\Lambda^{\ast}(r,\lambda)=r^{\ast}r+\lambda^{\ast}\cdot\lambda$ for some $r^{\ast}\in\bbR$ and $\lambda^{\ast}\in\bbR^{m}$. Since $(0,0)\in A\cap \overline{B}$, $\gamma=0$. Thus $r^{\ast}r+\lambda^{\ast}\cdot\lambda<0$ for every $r<0$ and $\lambda<0$, which can hold if and only if $r^{\ast}\geq 0$ and $\lambda^{\ast}\geq 0$. 

Now, for every $\eta\in \calC$, $(d\calP(\etastar;\eta-\etastar), \calG(\etastar)+d\calG(\etastar;\eta-\etastar))\in A$, and therefore
\begin{equation}\label{prop:kkt:eq1}
r^{\ast}d\calP(\etastar;\eta-\etastar)+\lambda^{\ast}\cdot ( \calG(\etastar)+d\calG(\etastar;\eta-\etastar)) \geq 0, \quad \eta\in \calC.
\end{equation}
In particular, by choosing $\eta=\etastar$, we obtain $\lambda^{\ast}\cdot \calG(\etastar)\geq 0$. However, the inequalities $\calG(\etastar)\leq 0$ and $\lambda^{\ast}\geq 0$ together imply $\lambda^{\ast}\cdot \calG(\etastar)\leq 0$, so that necessarily $\lambda^{\ast}\cdot \calG(\etastar)= 0$. Hence, \eqref{prop:kkt:eq1} becomes 
equation \eqref{eq:positive_lagrange}. If there exists $\eta\in \calC$ such that $\calG(\etastar)+d\calG(\etastar;\eta-\etastar)<0$, we cannot have $r^{\ast}=0$, for otherwise we would have $\lambda^{\ast}\neq 0$ and $\lambda^{\ast}\cdot d\calG(\etastar;\eta-\etastar)<0$, contradicting \eqref{eq:positive_lagrange}. In this situation we can therefore replace $\Lambda^{\ast}$ by $\Lambda^{\ast}/r^{\ast}$, and proceed in the same manner. Finally, we observe that this argument does not hold if $\calG$ is constant on $\calC$. However, in this scenario, $d\calG(\etastar;\eta-\etastar)=0$ for every $\eta\in \calC$ and\[
d\calP(\etastar;\eta-\etastar) = \lim\limits_{t\to 0^{+}}\frac{\calP((1-t)\etastar+t\eta)-\calP(\etastar)}{t}\geq 0
\]
for every $\eta\in \calC$, by convexity of $\calC$. Thus, \eqref{eq:positive_lagrange} holds
for all $\lambda^{\ast}\geq 0$.
\end{proof}

{\color{black}In several situations, $\calP$ and $\calG$ are regular enough so that their Gateaux derivatives satisfy assumptions stronger than convex linearity. For example, since the functionals are defined on a space of measures, these derivatives are often integration against a function, i.e., they are integral operators. This is still true for, e.g., the Value-at-Risk, as long as we restrict the functional to nicely behaved subsets like $\calC$ above. The following proposition makes use of this idea and specializes the results from Proposition~\ref{prop:kkt} to the scenario where $d\calP$ and $d\calG$ are effectively integral operators. }

\begin{proposition}\label{prop:support_minimizer}
In the setup of Proposition~\ref{prop:kkt}, {\color{black}assume further that $d\calP$ and $d\calG$ are integral operators on $\calC-\eta^*$ with continuous kernels, i.e., there exist continuous functions $p_{\etastar}:\bbR^{n}_{+}\times\bbR^{n}_{+}\to\bbR$ and $g_{\etastar}:\bbR^{n}_{+}\times\bbR^{n}_{+}\to\bbR^{m}$ such that
\[
d\calP(\etastar;\eta) = \int_{\bbR^{n}_{+}\times\bbR^{n}_{+}}p_{\etastar}(x,y)\eta(dx,dy)\quad \text{and}\quad d\calG(\etastar;\eta) = \int_{\bbR^{n}_{+}\times\bbR^{n}_{+}}g_{\etastar}(x,y)\eta(dx,dy)
\]
for $\eta\in \calC-\etastar$. Assume further that $p_{\etastar}$ and $g_{\etastar}$ 
are integrable w.r.t. elements in $\scrM-\etastar$.} Then, for every $(x,y)\in \mathrm{Supp}(\etastar)$, there exists a closed set $I\subset [0,x]$ with $y\in I$ and
\begin{equation}\label{eq:support_minimum}
r^{\ast}p_{\etastar}(x,y)+\lambda^{\ast}\cdot g_{\etastar}(x,y) = \min\limits_{t\in I} r^{\ast}p_{\etastar}(x,t)+\lambda^{\ast}\cdot g_{\etastar}(x,t).
\end{equation}
\end{proposition}
Here, for $x,y\in \bbR^{n}$, $[x,y]$ represents the closed box $[x_{1},y_{1}]\times\cdots\times [x_{n},y_{n}]$. Similarly, we denote by $]x,y[$ the product of the open intervals $]x_{1},y_{1}[\times\cdots\times ]x_{n},y_{n}[$.

{\color{black}Proposition~\ref{prop:support_minimizer} simply translates the information contained in \eqref{eq:positive_lagrange} into \eqref{eq:support_minimum}. The advantage of this is that \eqref{eq:positive_lagrange} is phrased in terms of \textit{measures} while \eqref{eq:support_minimum} is given in terms of \textit{points} in $\bbR^{n}\times\bbR^{n}$, which are usually easier to handle. This translation is done by describing the support of optimal contracts in terms of the functions $p_{\etastar}$ and $g_{\etastar}$. Notice that a more direct way of going from measures to points would be by evaluating \eqref{eq:positive_lagrange} at the Dirac measures $\delta_{(x,y)}$, however, by absolute continuity of $\mu$, $\delta_{(x,y)}\not\in\scrM$ 
for every $(x,y)\in\calA$, so another approach has to be used.}
\begin{proof}
Let $(x,y)$ be an arbitrary point in $\calA_{R}$. For $t\in\bbR^{n}$ and $\varepsilon>0$,
define the measure $\eta_{x,y,t,\varepsilon}$ by
\[
\eta_{x,y,t,\varepsilon}(A) = \etastar(A)-\etastar(A\cap B_{\varepsilon}(x,y))+\etastar((A-(0,t))\cap B_{\varepsilon}(x,y)),\quad \text{for $A\in \bbR^{n}_{+}\times \bbR^{n}_{+}$},
\]
where $A-(0,t)$ denotes the translation of $A$ by $-(0,t)\in \bbR^{n}_{+}\times \bbR^{n}_{+}$, and $B_{\varepsilon}(x,y)$ is the open ball around $(x,y)$ of radius $\varepsilon$. The proof will be divided into four steps.\\[0.2cm]
\textit{Step 1. Let $(x,y)\in\calA_R$ be such that $0<y<x$. Define $\delta$ by
\[
\delta=\min\left\{d((x,y),\partial \calA_{R}),d((x,y),\partial \calA_{R}-(0,t))\right\}.
\]
Then $\delta>0$ and $\eta_{x,y,t,\varepsilon}\in \scrM$ for every $-y<t<x-y$ and $0\leq \varepsilon<\delta$.
}\\[0.2cm]
Here, $d((x,y),B)$ denotes the distance of $(x,y)$ from the set $B$. The idea of this choice is that $\delta$ is the largest number for which the last two terms in the definition of $\eta_{x,y,t,\varepsilon}(\calA_R)$ cancel out, thus producing an element of $\scrM$. Now, strict positivity of $\delta$ follows from the choice of $y$ and $t$: since $0<y<x$, $(x,y)\not\in \partial \calA_{R}$, and since $(x,y)\in \calA_{R}-(0,t)$ if and only if $(x,y+t)\in \calA_{R}$, it follows that, for $-y<t<x-y$, $(x,y+t)\not\in \partial \calA_{R}$, so $(x,y)\not\in \partial \calA_{R}-(0,t)$. As both $\partial \calA_{R}$ and $\partial \calA_{R} -(0,t)$ are closed, we obtain $\delta>0$. Next, we need to check that $\eta_{x,y,t,\varepsilon}\in \scrM$. Since $\eta_{x,y,t,\varepsilon}$ is a (positive) measure in $\bbR^{n}_{+}\times \bbR^{n}_{+}$, we just need to show that it is a probability measure giving full measure to $\calA_{R}$ with ${\pi_1}_\#\eta_{x,y,t,\varepsilon}=\mu$. We have:
\begin{itemize}
    \item $\eta_{x,y,t,\varepsilon}(\calA_{R})=1$ and $\eta_{x,y,t,\varepsilon}(\calA_{R}^{c})=0$, since $B_{\varepsilon}(x,y)\subset \calA_{R} \cap (\calA_{R} - (0,t))$ for $0\leq \varepsilon<\delta$;
    \item For measurable $A\subset \bbR^{n}_{+}$, $A\times \bbR^{n} - (0,t) = A\times \bbR^{n}$, so
    \[
        \eta_{x,y,t,\varepsilon}(A\times \bbR^{n}) = \etastar(A\times \bbR^{n}) = \mu(A).
    \]
\end{itemize}
We conclude that $\eta_{x,y,t,\varepsilon}\in \scrM$.\\[0.2cm]
\textit{Step 2. If $(x,y)\in \mathrm{Supp}(\etastar)$ is such that $0<y<x$, then, for every $-y<t<x-y$ and continuous function $f:\bbR^{n}_{+}\times \bbR^{n}_{+}\to \bbR^{n}$, we have
\begin{equation}\label{eq:convergence_weak}
\lim\limits_{\varepsilon\to 0} \int f d\eta_\varepsilon = f(x,y+t)-f(x,y),
\end{equation}
where 
\[
\eta_\varepsilon = \frac{\eta_{x,y,t,\varepsilon}-\etastar}{\etastar(B_{\varepsilon}(x,y))}.
\]
}\\[0.05cm]
Let $\vartheta_\varepsilon,\pi_\varepsilon\in \scrP(\bbR^{n}_{+}\times \bbR^{n}_{+})$ be given by
\[
    \vartheta_\varepsilon(A) = \frac{\etastar(A\cap B_{\varepsilon}(x,y))}{\etastar(B_{\varepsilon}(x,y))}
\]
and
\[
    \pi_\varepsilon(A) = \frac{\etastar((A-(0,t))\cap B_{\varepsilon}(x,y))}{\etastar(B_{\varepsilon}(x,y))},
\]
so that $\eta_\varepsilon=\pi_\varepsilon-\vartheta_\varepsilon$. To prove \eqref{eq:convergence_weak} it is then enough to show that $\int f d\vartheta_\varepsilon \to f(x,y)$ and $\int f d\pi_\varepsilon \to f(x,y+t)$. For this, note that
\[
\left|\int f d\vartheta_\varepsilon -f(x,y)\right| = \left|\frac{1}{\etastar(B_{\varepsilon}(x,y))} \int_{B_{\varepsilon}(x,y)} f d\etastar - {\color{black} f(x,y)}\right|\to 0\qquad \text{as $\varepsilon\to 0$}
\]
by the Lebesgue differentiation theorem. The other convergence is proved similarly.\\[0.2cm]
\textit{Step 3. For $(x,y)\in \mathrm{Supp}(\etastar)$ with $0<y<x$, the statement of the proposition holds by taking 
\[
I = \overline{\{t\in ]0,x[\ \mid 0\text{ is a limit point of } \{\varepsilon\in [0,\delta) \mid \eta_{x,y,t,\varepsilon}\in \calC\}\}}.
\]
}\\[0.05cm]
Define
\[
E(t) = \{\varepsilon\in [0,\delta) \mid \eta_{x,y,t,\varepsilon}\in \calC\}
\]
and
\[
J=\{t\in ]0,x[\ \mid 0\text{ is a limit point of } E(t-y)\}.
\]
By assumption, we know that $\etastar\in \calC$, so $E(0) =[0,\delta)$, implying $y\in J$. 
From \eqref{eq:positive_lagrange} and linearity, we obtain
\[
\int_{\bbR^{n}_{+}\times\bbR^{n}_{+}}r^{\ast}p_{\etastar}(x,t)+\lambda^{\ast}\cdot g_{\etastar}(x,t) \; d\eta_\varepsilon \geq 0
\]
for $t\in J$ and $\varepsilon\in E(t-y)$. Letting $\varepsilon\to 0$ in the previous inequality, by Step 2 we obtain
\[
r^{\ast}(p_{\etastar}(x,t)-p_{\etastar}(x,y))+\lambda^{\ast}\cdot (g_{\etastar}(x,t)-g_{\etastar}(x,y)) \geq 0, \quad t\in J(x,y).
\]
Using continuity once again, we see that this inequality is valid for every $t$ in the closure of $J$, so we can take $I=\overline{J}$, obtaining \eqref{eq:support_minimum}.\\[0.2cm]
\textit{Step 4. The statement of the proposition holds for arbitrary $(x,y)\in \mathrm{Supp}(\etastar)$.
}\\[0.2cm]
Observe that it only remains to show that the statement is true for $(x,y)\in \partial\calA_R$. Step~1 still holds by setting $\delta=d((x_0,y_0),\partial \calA_{R}-(0,t))$ for $-y<t<x-y$. Notice that now we only have $B_{\varepsilon}(x,y)\cap \subset \calA_{R} \cap (\calA_{R} - (0,t))$. However, this is enough to show that $\eta_{x,y,t,\varepsilon}(\calA_{R})=1$ and $\eta_{x,y,t,\varepsilon}(\calA_{R}^{c})=0$ by noticing that 
\[
\etastar(A) = \etastar(A\cap \calA_{R})
\]
for every measurable $A\subset \bbR^{n}_{+}$, so that we anyway have  $\eta_{x,y,t,\varepsilon}\in \scrM$. Steps 2 and 3 carry over verbatim to this case.
\end{proof}
\begin{remark}\normalfont The measures $\eta_{x,y,t,\varepsilon}$ in the previous proof were already used in \cite{guerra2012quantile} and again in \cite{guerra2021reinsurance} with the same objective of identifying the support of optimal reinsurance contracts. However, since the functionals in these papers were more specific, one was able to directly obtain a description of the sets $I$ there in a more direct manner. \hfill $\diamond$ 
\end{remark}
\begin{remark}\normalfont
Note that the conclusion of Proposition~\ref{prop:support_minimizer} is somewhat redundant for our purposes: Equation \eqref{eq:support_minimum} assumes knowledge of points in the support of $\etastar$, which is the contract that we wish to determine. Nonetheless, the procedure can be reversed: if by some mechanism we can identify the sets $I$, then we will know that points in the support of $\etastar$ will be the minima of $r^{\ast}p_{\etastar}(x,\cdot)+\lambda^{\ast}\cdot g_{\etastar}(x,\cdot)$.  In most situations, $p_{\etastar}$ and $g_{\etastar}$ depend on $\etastar$ only through a finite set of parameters, so we can compute these minima without full specification of $\etastar$, and in common applications, there will be only finitely many minima. Hence, the task of finding optimal contracts is then reduced to identifying the minima that \textit{do} belong to the support of $\etastar$ together with the correct parameters specifying $\etastar$, a task that is in general easier than computing a full measure\footnote{This can again be compared with the common optimization technique in $\bbR^{N}$: when looking for extremes of a smooth function, one first finds the zeros of the gradient and then chooses the zeros that produce the desired extrema.}. Examples~\ref{ex:ex2-1:expectation_and_variance} and \ref{ex:ex2-3:VaR_and_expectation} below illustrate the procedure of identifying the minima, reducing the problem to a two-dimensional and one-dimensional optimization task, respectively. \hfill $\diamond$ 
\end{remark}

When $\calP$ and $\calG$ are regular across the whole of $\scrM$, more can be said.
\begin{proposition}\label{prop:suppport_minimum}
In the setting of Proposition~\ref{prop:support_minimizer}, assume we can take $\calC=\calD=\scrM$. Moreover, assume that the partial minimization function
\[
m(x) = \inf_{y\in [0,x]} r^{\ast}p_{\etastar}(x,y)+\lambda^{\ast}\cdot g_{\etastar}(x,y)
\]
is measurable. Then 
\begin{equation}\label{eq:support_minimum_II}
\etastar\left (\{(x,y)\in \calA_{R}\mid y \in \argmin\limits_{t\in [0,x]} r^{\ast}p_{\etastar}(x,t)+\lambda^{\ast}\cdot g_{\etastar}(x,t)\} \right) = 1.
\end{equation}
\end{proposition}
\begin{proof}
Let $h:\bbR^{n}_{+}\times \bbR^{n}_{+}\to \bbR$ be given by $h(x,y)=r^{\ast}p_{\etastar}(x,y)+\lambda^{\ast}\cdot g_{\etastar}(x,y)$ and let $M$ be the set appearing in \eqref{eq:support_minimum_II}. Observe first that 
\[
M = \{(x,y)\in \calA_{R}\mid h(x,y)-m(x)=0\},
\]
so measurability of $m$ implies that $M$ is measurable, and \eqref{eq:support_minimum_II} makes sense. The statement then follows once we show the inequality
\begin{equation}\label{prop:suppport_minimum:eq:minimum_integral}
    \int_{\bbR^{n}_{+}\times \bbR^{n}_{+}} h(x,y) \; \etastar(dx,dy) \leq \int_{\bbR^{n}_{+}}m(x)\; \mu(dx).
\end{equation}
Indeed, since $\int_{\bbR^{n}_{+}\times \bbR^{n}_{+}} m(x) \ \etastar(dx,dy) = \int_{\bbR^{n}_{+}}m(x)\ \mu(dx)$ and $h(x,y)\geq m(x)$ for every $(x,y)\in \calA_{R}$, from the inequality in \eqref{prop:suppport_minimum:eq:minimum_integral} it follows that even equality holds, thus $h(x,y)= m(x)$ $\etastar$-a.e., which is equivalent to \eqref{eq:support_minimum_II}.

To prove \eqref{prop:suppport_minimum:eq:minimum_integral}, we make use of Theorem 5.5.3 in \cite{srivastava2008course}, which states that every analytic subset $A$ of the product of two Polish spaces admits a section $s$ that is universally measurable, i.e., $s$ is measurable with respect to the completion of the Borel $\sigma$-algebra w.r.t.\ any probability measure. With $A=M$, it follows that there exists a function $s:\bbR^{n}\to \bbR^{n}\times \bbR^{n}$ such that $s(x)\in M$ and $\pi_{1}\circ s (x)=x$ for every $x\in \bbR^{n}$ (a \textit{section} of $M$), and such that $s$ is universally measurable. In particular, $s$ is measurable with respect to the $\mu$-completion of $\bbR^{n}$. Let $\eta$ be the probability measure on $\bbR^{n}\times \bbR^{n}$ given by $\eta=s_{\#}\mu$. Notice that by the $\mu$-measurability of $s$, $\eta$ is well defined and, moreover, it is a reinsurance contract. Hence, \eqref{eq:positive_lagrange}  
implies
\begin{multline*}
    r^{\ast}d\calP(\etastar;\eta-\etastar)+\lambda^{\ast}\cdot d\calG(\etastar;\eta-\etastar)\\  =  r^{\ast}\int_{\bbR^{n}_{+}\times \bbR^{n}_{+}} p_{\etastar}(x,y)(\eta-\etastar)(dx,dy)+ \lambda^{\ast}\int_{\bbR^{n}_{+}\times \bbR^{n}_{+}} g_{\etastar}(x,y)(\eta-\etastar)(dx,dy)\\
     = \int_{\bbR^{n}_{+}\times \bbR^{n}_{+}} h(x,y)(\eta-\etastar)(dx,dy)\geq 0,
\end{multline*}
so
\begin{align*}
    \int_{\bbR^{n}_{+}\times \bbR^{n}_{+}} h(x,y) \; \etastar(dx,dy)&\leq \int_{\bbR^{n}_{+}\times \bbR^{n}_{+}} h(x,y) \; \eta(dx,dy)\\
    &= \int_{\bbR^{n}_{+}} h\circ s(x) \; \mu(dx) = \int_{\bbR^{n}_{+}}m(x)\; \mu(dx),
\end{align*}
as desired.
\end{proof}

\begin{corollary}\label{cor:prop:suppport_minimum}
In the setting of Proposition~\ref{prop:suppport_minimum}, let $\calG=0$ and $\calP$ be given by
\[
\calP(\eta)= \int_{\bbR^{n}_{+}\times \bbR^{n}_{+}} p(x,y) \eta(dx,dy),
\]
for a continuous function $p:\bbR^{n}_{+}\times \bbR^{n}_{+}\to \bbR$, and define the set $M$ by
\[
M = \{(x,y)\in \calA_{R}\mid y \in \argmin\limits_{t\in [0,x]} p(x,t)\}.
\]
Then, $\eta\in \scrM$ is an optimal reinsurance contract if and only if $\eta(M)=1$.
\end{corollary}
\begin{proof}
Let $\eta\in \scrM$. If $\eta$ is optimal, Proposition~\ref{prop:suppport_minimum} implies $\eta(M)=1$, since $h=p_{\eta}=p$ and $g_\eta=0$ (i.e., the function $p_\eta$ is the same for all optimal contracts), so we only need to prove the reverse implication. However, this is immediate from (the proof of) Proposition~\ref{prop:suppport_minimum}, since if $\eta(M)=1$, then $p(x,y)=m(x)$ $\eta$-a.e., where $m$ is the partial minimization function, so
\begin{align*}
    \calP(\eta) &= \int_{\bbR^{n}_{+}\times \bbR^{n}_{+}} p(x,y) \; \eta(dx,dy) = \int_{\bbR^{n}_{+}\times \bbR^{n}_{+}} m(x) \; \mu(dx)\\
    &= \int_{\bbR^{n}_{+}\times \bbR^{n}_{+}} p(x,y) \; \etastar(dx,dy) = \calP(\etastar),
\end{align*}
where $\etastar$ is any optimal reinsurance contract. Hence $\eta$ is optimal.
\end{proof}
\begin{remark}\normalfont
In Proposition~\ref{prop:suppport_minimum}, it is necessary to require the partial minimization function to be measurable. While this might seem like an assumption that should usually be satisfied, in general the partial minimization operation yields only a lower semi-analytic function and one cannot ensure the set $M$ to be analytic (see Proposition 7.47 of \cite{bertsekas1996stochastic} for further details). This observation also explains why tools from descriptive set theory need to be used in a place which is seemingly unrelated. Note, however, that these tools are mostly required to ensure the existence of the section $s$. In some applications, the functions $p$ and $g$ are smooth enough so that one can show the existence of $s$ by more conventional methods (e.g., the implicit function theorem). \hfill $\diamond$ 
\end{remark}
\begin{remark}\normalfont
If for every $x\in\bbR^n$ the function $r^{\ast}p_{\etastar}(x,\cdot)+\lambda^{\ast}\cdot g_{\etastar}(x,\cdot)$ has exactly one minimizer, the section $s$ from the proof of Proposition~\ref{prop:suppport_minimum} is unique $\mu$-a.e and $\etastar = (\mathrm{Id},f)_\#(\mu)$ for some function $f$. 
However, it does not necessarily follow that a deterministic reinsurance contract is optimal, since $f$ might fail to be Borel-measurable. This might happen when, for instance, $f$ does not ``move measurably'' from one $x$ to the other, describing a measurable set of measure zero in $\bbR^n\times\bbR^n$ without a measurable projection into $\bbR^n$. For example, for $n=1$, let $E$ be a non-Lebesgue-measurable subset of $]0,\infty[$ and define
\[
E' = \{(x,y)\in\bbR^{2}_{+}\mid y = x \text{ if } x\in E\text{ and } y=0 \text{ otherwise}\}\subset D\cup (\bbR_{+}\times \{0\}),
\]
where $D=\{(x,x)\in\bbR^{2}_{+}\mid x\geq 0\}$ is the diagonal of $\bbR^{2}_{+}$. As $D\cup (\bbR_{+}\times \{0\})$ has Lebesgue-measure zero, $E'$ is Lebesgue measurable. Letting $p:\bbR_{+}\times\bbR_{+}\to \bbR_{+}$ be the distance to $E'$, $\calP(\eta) = \int p\;d\eta$ and $G=0$, we see that $p$ is continuous and the assumptions from Proposition~\ref{prop:suppport_minimum} are satisfied with partial minimization function identically zero. Moreover, $r^{\ast}p_{\etastar}(x,\cdot)+\lambda^{\ast}\cdot g_{\etastar}(x,\cdot) = p(x,\cdot)$ has exactly one minimum, namely, at $0$ or $x$,
so that $f(x) = x1_E(x)$, where $1_E$ is the indicator function of $E$. Hence, $f$ is not measurable, and so, by uniqueness, no deterministic reinsurance contract exists. Observe, however, that any contract supported on $\overline{E'}$ is optimal, so one can easily identify optimal contracts.
It is clear that this kind of pathology is not likely to appear in examples commonly happening in practice, however the fact that this kind of risk measure is considered within our assumptions points towards the generality of our setting. \hfill $\diamond$ 
\end{remark}
\begin{remark}\normalfont
Corollary~\ref{cor:prop:suppport_minimum} is, in a way, the best we can do in terms of fully identifying optimal reinsurance contracts. Coming back to the hypotheses and notation of Proposition~\ref{prop:suppport_minimum}, the same argument given in the proof of Corollary~\ref{cor:prop:suppport_minimum} shows that
\[
\int_{\bbR^{n}_{+}\times \bbR^{n}_{+}} h(x,y) \; \etastar(dx,dy)= \int_{\bbR^{n}_{+}\times \bbR^{n}_{+}} h(x,y) \; \hat{\eta}(dx,dy)
\]
for any $\hat{\eta}\in \scrM$ such that $\hat{\eta}(M)=1$. This implies that an equation analogous to \eqref{eq:positive_lagrange} is valid when we replace $\etastar$ by $\hat{\eta}$ on the second argument of $d\calP$ and $d\calG$, i.e.,
\[
r^{\ast}d\calP(\etastar;\eta-\hat{\eta})+\lambda^{\ast}\cdot d\calG(\etastar;\eta-\hat{\eta}) \geq 0, \quad \eta\in\scrM.
\]
However, this condition is not sufficient for optimality when $\calG$ is not constant. \hfill $\diamond$ 
\end{remark}

\section{\color{black}More general constraints}\label{secOT}


{\color{black}In the previous section it was necessary that the set $\calS$ was described by a finite set of inequalities (finitely many constraints). In this section we drop this assumption and investigate which conclusions can be drawn when the constraints are more general (that is, possibly involving equalities or infinitely many constraints).\footnote{Recall from Remark~\ref{remthis} that introducing equalities through $g\le 0$ and $-g\le 0$ was not necessarily possible for positive $r^*$.} As indicated before, this is a rather hard task if we allow $\calP$ to be an arbitrary (lower semi-continuous) functional. 
We then make some assumptions, inspired by the idea of local linearization from the previous sections, which will allow us to express problem \eqref{eq.prob} in a similar fashion as in optimal transport.}

\begin{assumption}\label{assump:A1}
If $\etastar\in \calS$ is an optimal reinsurance contract, then, for every $\eta\in \calS$ and $0\leq t\leq 1$, we have 
\[
\calP(\eta^{\ast})\leq \calP((1-t)\eta^{\ast}+t\eta).
\]
\end{assumption}
\begin{assumption}\label{assump:A2}
For every $\eta\in \calS$, $d\calP(\eta;\cdot)$ exists for every direction in $\calS-\eta$ and is given as an integral operator, i.e., there exists a measurable function $p_{\eta}:\bbR^{n}\times\bbR^{n}\to\bbR$ such that, for every $\vartheta\in \calS$,
\[
d\calP(\eta;\vartheta-\eta) = \int p_{\eta}(x,y)(\vartheta-\eta)(dx,dy).
\]
\end{assumption}
Observe that Assumption \ref{assump:A2} implicitly assumes that the integral is finite. Given the considerations made in the previous sections, Assumption \ref{assump:A2} seems natural in order to linearize the problem. In contrast, Assumption \ref{assump:A1} might seem odd, as it may appear more natural to require $\calS$ to be convex instead. In our examples, this will most often be the case, but we prefer to phrase it this way to cover a larger amount of scenarios (for example, when $\calS$ is arbitrary and $\calP$ is concave, as is the case for Value-at-Risk in Example~\ref{ex:ex2}). The next proposition illustrates how this two requirements allow us to see the problem from a different perspective.

{\color{black}
\begin{proposition}
    Assumptions \ref{assump:A1} and \ref{assump:A2} jointly imply that, if $\etastar$ is an optimal reinsurance contract, then
\begin{equation}\label{eq:minimum_etastar_OT}
\int \qeta(x,y)\,\eta^{\ast}(dx,dy)= \min\limits_{\nu \in \pi_2(\calS)}\calC(\mu,\nu),
\end{equation} 
where $\qeta$ denotes the function on $\bbR^{n}\times\bbR^{n}$ such that $\qeta(x,y)=\peta(x,y)$ on $\calA_{R}$ and otherwise being equal to $+\infty$, $\pi_2(\calS)=\{{\pi_2}_\#\eta : \eta\in\calS\}$ and
\begin{equation}\label{eq:OT_problem}
\calC(\mu,\nu)= \min\limits_{\eta \in \Pi(\mu,\nu)\cap \calS} \int \qeta(x,y)\,\eta(dx,dy).
\end{equation} 
\end{proposition}
\begin{proof}
    Let $\etastar$ be an optimal contract and $\eta\in \calS$. Assumption \ref{assump:A1} implies
\begin{align*}
    0\leq \lim\limits_{t\to 0^+}\frac{ \calP((1-t)\eta^{\ast}+t\eta)-\calP(\eta^{\ast})}{t} = d\calP(\etastar;\eta-\etastar),
\end{align*}
so that, by Assumption \ref{assump:A2}, $\int \peta(x,y)\,\eta^{\ast}(dx,dy)\leq \int \peta(x,y)\,\eta(dx,dy)$. As $\eta$ was arbitrary and $\etastar\in\calS$, it follows that 
\begin{equation}\label{eq:minimum_etastar}
\int \peta(x,y)\,\eta^{\ast}(dx,dy)= \min\limits_{\eta \in \calS} \int \peta(x,y)\,\eta(dx,dy).
\end{equation} 
As all contracts are supported on $\calA_{R}$, we can replace $\peta$ by $\qeta$ in the integrals and by casting the minimum in \eqref{eq:minimum_etastar} as a double minimum, the proposition follows.
\end{proof}
}

{\color{black} The meaning of this is that minimizers of $\calP$ are solutions of \eqref{eq:minimum_etastar}, which gives a necessary condition for optimal contracts.} Equations \eqref{eq:minimum_etastar_OT}-\eqref{eq:OT_problem} mean that the optimal contract satisfies a double minimization property, where the inner minimum \eqref{eq:OT_problem} is a constrained optimal transport problem (the couplings need to satisfy the constraint of belonging to $\calS$). Note that we still face the issue from the previous section that the function $\peta$ depends on $\etastar$, and so does the cost function in \eqref{eq:OT_problem}. This means that we are facing transport problems depending on the optimal treaty that we are looking for. The idea behind \eqref{eq:OT_problem} is, however, similar to the one developed before in the sense that, for a large set of functionals, the function $\qeta$ will depend on $\etastar$ solely through a finite set of parameters which can be thought of as fixed at the beginning. The hope is that, by means of optimal transport techniques, one can provide information about the general structure of optimal couplings $\etastar$, for example about the geometric characterizations of their supports, and by leveraging this one can find the parameters which achieve the minimum in \eqref{eq:minimum_etastar_OT}.\footnote{Note, however, that this is not a requirement for optimality but rather a consequence. After having used the information to characterize optimal treaties, one still needs to find the one that minimizes $\calP$.} For example, if for every $\nu\in\pi_{2}(\calS)$ there exists an optimal Monge coupling for the OT problem \eqref{eq:OT_problem}, then $\etastar$ is also given by a deterministic reinsurance contract. This is evident by observing that, for an optimal reinsurance treaty $\etastar$, we have
\begin{equation*}
\int \qeta(x,y)\,\eta^{\ast}(dx,dy)=\calC(\mu,\nustar),
\end{equation*} 
where $\nustar={\pi_2}_\#\etastar$. Now, while existence of an optimal Monge coupling is a rather scarce property in OT problems, this observation is relevant enough to cover some interesting cases in the context of optimal reinsurance. This is illustrated in the following proposition, and then applied in Examples~\ref{ex:ex1} and \ref{ex:ex2} below.

\begin{proposition}\label{prop:monotonic_rea}
Assume $\calP$ is given by
\begin{equation}
\calP(\eta)=\calP_{1}(T_\#\eta)+\calP_{2}({\pi_2}_\#\eta),
\end{equation}
where $T$ is the linear operator $T:\mathbb{R}^{n}\times \mathbb{R}^{n}\to \bbR^{n}$ given by $T(x,y)=x-y$, and the functionals $\eta\mapsto \calP_{1}(T_\#\eta), \eta\mapsto\calP_{2}({\pi_2}_\#\eta)$ from $\scrM$ to $\bbR$ satisfy Assumption \ref{assump:A2} with functions $h_{\eta}, k_{\eta}:\bbR^{n}\to\bbR$, respectively. Then, if $\calP$ Assumption \ref{assump:A1}, it also satisfies and Assumption \ref{assump:A2} with $p_{\eta}$  given by 
\begin{equation}\label{prop:monotonic_rea:eq:cetastar}
p_{\eta}(x,y) = h_{\eta}(x-y) + k_{\eta}(y).
\end{equation}
Moreover, if $n=1$ and
\begin{itemize}
\item[(i)] for every $\nu \in \pi_{2}(\calS)$, we have $\Pi(\mu,\nu)\cap \scrM\subset \calS$,
\item[(ii)] $h_{\etastar}$ is strictly convex for every optimal treaty $\etastar$, and
\item[(iii)] the distribution of $X$ is continuous,
\end{itemize}
then a unique optimal Monge map exists for \eqref{eq:OT_problem} for every $\nu \in \pi_{2}(\calS)$. {\color{black}In particular, every optimal reinsurance contract is deterministic.}
\end{proposition}
\begin{proof}[Proof of Proposition~\ref{prop:monotonic_rea}]
It is easy to see that $\calP$ satisfies (A2) and $p_{\eta}$ is given as in \eqref{prop:monotonic_rea:eq:cetastar}, so we only need to show the second part of the statement. 
For this, fix $\nu \in \pi_{2}(\calS)$ and observe that by (A2), $\calC(\mu,\nu)$ is finite. The first condition means that the constraints depend only on the second marginal of any reinsurance treaty, so that for any optimal contract $\etastar$,
\begin{equation}\label{eq:min_peta_to_heta}
\begin{split}
       \min\limits_{\eta \in \Pi(\mu,\nu)\cap \calS} \int \peta(x,y)\,\eta(dx,dy)&=\min\limits_{\eta \in \Pi(\mu,\nu)} \int \peta (x,y)\,\eta(dx,dy) \\
&= c(\nu) +  \min\limits_{\eta \in \Pi(\mu,\nu)} \int h_{\etastar}(x-y)\,\eta(dx,dy) ,
\end{split}
\end{equation}
where $c(\nu)=\int_0^{\infty} k_{\etastar}(y)\,\nu(dy)$. The second and third condition from the statement of the proposition are technical conditions that ensure the existence of a unique optimal Monge map for the last minimum at the end of \eqref{eq:min_peta_to_heta}, given by
{\color{black}
\[
g_\nu = F_{\nu}^{-1}\circ F_{\mu},
\]
known as the Brenier map between $\mu$ and $\nu$; see \cite{villani2021topics}.} 
Since $\nu\prec_1\mu$, we have $g(x)\leq x$ for every $x\geq 0$, so $g$ {\color{black}(or better, $(\mathrm{Id},g_\nu)_\#\mu$)} is also an optimizer for $\calC(\mu,\nu)$. 
Since $\nu$ was arbitrary, the first part of the proposition follows. {\color{black}Finally, if $\etastar$ is an optimal contract, it is a solution of \eqref{eq:minimum_etastar}, so by uniqueness it follows that 
\begin{equation*}
\etastar=(\mathrm{Id},g_{\nustar})_\#\mu,
\end{equation*}
where $\nustar = {\pi_{2}}_\#\etastar$.
}
\end{proof}

\begin{remark}\label{remark:monge_1}
\normalfont
{\color{black} Notice that the assumption $n=1$ is crucial in the proof of Proposition~\ref{prop:monotonic_rea}. For general $n$, the Gangbo-McCann theorem implies that the minimum of the OT problem with cost $p_\etastar$ is achieved by an unique optimal Monge coupling $\pi^{\ast}\in \Pi(\mu,\nu)$ of the form 
\begin{equation*}\label{gm}
\pi^*=(\mathrm{Id},\tau^*)_\#\mu=(\mathrm{Id},\mathrm{Id}-(\nabla \peta)^{-1}\circ\nabla\varphi)_\#\mu,
\end{equation*}
for some $\peta$-concave function $\varphi$. These optimizers are optimal reinsurance contracts if and only if
\begin{equation*}\label{eq:gambo_is_rc}
    0\leq (\nabla \peta)^{-1}\circ\nabla\varphi\leq \mathrm{Id}.
\end{equation*}
The proof of the second part of Proposition~\ref{prop:monotonic_rea} uses the fact that, for $n=1$, the condition $\nu^\ast\prec_1\mu$ is enough to make \eqref{eq:gambo_is_rc} hold. However, this is not necessarily true in higher dimensions. Therefore, obtaining conditions on $\nustar$ ensuring that  \eqref{eq:gambo_is_rc} holds could be a direction for further research in the area.
}
\hfill $\diamond$ 
\end{remark}
In what follows, we use the notation $\pi_{i,j}$ for the composition $\bbR^{n}\times \bbR^{n}\to \bbR^{n}\to \bbR$ of the projection onto the $i^{th}$-coordinate, $i=1,2$, and then onto the $j$-th coordinate, $j=1,\ldots,n$.

\begin{proposition}\label{prop:factorization_marginals}
Assume that $\calP$ can be written as a composition $\calP=\calF\circ\calR$, where $\calR:\scrM\to \scrP(\bbR)^{n}$ is given by
\[
\calR(\eta)=({\pi_{2,1}}_\#\eta,\ldots,{\pi_{2,n}}_\#\eta)
\]
and $\calF:\scrP(\bbR)^{n}\to \bbR$ is a lower semi-continuous functional. Assume moreover that there exists a topological space $O$ such that $\calS$ is given as $\calS = \calU^{-1}(E)$ for some closed $E\subset O$, where $\calU = \calV\circ \calR$ and $\calV:\scrP(\bbR)^{n}\to O$ is continuous on the image of $\calR$. Let $C$ be a copula for the distribution $\mu$ of $X$.  Then, if $\etastar$ is an optimal reinsurance treaty
and $\nustar_{i}={\pi_{2,i}}_\#\etastar$, $i=1,\ldots,n$, the treaty $\vartheta$ whose distribution function is given by
\[
F_{\pi}(x_{1},\ldots,x_{n},y_{1},\ldots,y_{n}) = C(\min(F_{\mu_1}(x_{1}),F_{\nustar_{1}}(y_{1})),\ldots, \min(F_{\mu_n}(x_{n}),F_{\nustar_{1}}(y_{n})))
\]
is also optimal. In particular, if each $F_{\mu_i}$ is continuous, an optimal deterministic reinsurance contract exists, where the components are given by the functions
\[
{\color{black}R_i(x) = F_{\nustar_{i}}^{-1}\circ F_{\mu_i}(x)},
\quad i=1,\ldots, n.
\]
\end{proposition}
\begin{proof}
This is straightforward after noticing that $\calR(\eta^{\ast})=\calR(\pi)$. Observe that the distributions $\min(F_{\mu_i},F_{\nustar_{i}})$ represent the c.d.f's of the monotonic rearrangements of the individual claims with the $\nustar_{i}$'s, so by the argument given at the end of the proof of Proposition~\ref{prop:monotonic_rea}, $\vartheta$ is concentrated on $\calA_{R}$.
\end{proof}

Proposition~\ref{prop:factorization_marginals} simply states that if the risk measure and the constraints depend solely on the reinsured distribution, there is \emph{no need for randomization}. Observe moreover that the statement does not make use of the particular form of $\calP$ other than its factorization, so the statement is valid if we change the point of view and assume that ${\pi_2}_{\#}\eta$ represents the distribution of the deductible instead of the reinsured amount. Proposition~\ref{prop:monotonic_rea} differs from Proposition~\ref{prop:factorization_marginals}, in the sense that Proposition~\ref{prop:monotonic_rea} allows for a ``mixing'' between the deductible and the reinsured amount, while Proposition~\ref{prop:factorization_marginals} assumes $\calP$ is only determined by the reinsured amounts.

\section{Examples} \label{sec5}
We explore the techniques developed in the previous sections to illustrate some examples in optimal reinsurance that can be approached with this framework; see \cite[Ch. VIII]{albrecher2017reinsurance} for a systematic survey on optimal reinsurance problems. As in Proposition~\ref{prop:monotonic_rea}, in the sequel the linear operator $T:\mathbb{R}^{n}\times \mathbb{R}^{n}\to \bbR^{n}$ represents the vector of retained risks $T(x,y)=x-y$. For any $\nu\in\scrP(\bbR)$, we denote by $\overline{\nu}$ the mean of $\nu$, that is, $\overline{\nu}=\int y \nu(dy)$.

The following three examples exemplify how the results of Section~\ref{sec3} can be used. For ease of reading, several details of the rigorous derivation are deferred to the appendix. 
\begin{example}\label{ex:ex2-1:expectation_and_variance}\normalfont
In this example we deal with a problem originally considered by de Finetti in \cite{definetti1940} (see also Section 8.2.6.1 in \cite{albrecher2017reinsurance}). Here, a first-line insurer has $n$ sub-portfolios with insurance risks $X_1,\ldots, X_n$ and is looking for a reinsurance contract $R=(R_1,\ldots,R_n)$ that minimizes the aggregate expected loss after reinsurance, under a constraint on the retained aggregate variance. Assume that the premium for the $i$-th contract is computed according to an expected value principle with safety loading $\beta_i$, so that the total loss experienced by the first line insurer is given by 
\[
\sum_{i=1}^{n} (X_i-R_i +(1+\beta_i)\bbE[R_i]).
\]
Consider the case where all the $\beta_i$'s are different (e.g.\ because they represent different business lines) and w.l.o.g.\ ordered increasingly, i.e., $0<\beta_1<\cdots < \beta_n$. Assume further a bound on the retained variance, i.e., that for a constant $c$ s.t. $0<c<\Var(\sum_{i=1}^n X_i)$, the contract is required to satisfy $\Var(\sum_{i=1}^n (X_i-R_i))\leq c$.

The risk measure $\calP$ can be taken to be
\begin{equation}\label{thatone}
\calP(\eta)=\int\sum_{i=1}^{n}\beta_{i}y_{i} \; \eta(dx,dy)
\end{equation}
while $\calG$ can then be written as
\[
\calG(\eta)= \int\left(\sum_{i=1}^{n}(x_{i}-y_{i})\right)^{2}-\left(\sum_{i=1}^{n}\int(x_{i}-y_{i})\;\eta(dx,dy)\right)^{2}\; \eta(dx,dy) - c.
\] 
Notice that $\calP$ is obtained after taking expectations on the total loss, while $\calG$ is simply $\Var(\sum_{i=1}^n (X_i-R_i))- c$ written in terms of the measure $\eta$.

By means of Propositions~\ref{prop:support_minimizer} and \ref{prop:suppport_minimum}, it can be seen that the optimal reinsurance contract is deterministic and component-wise given by
\begin{equation}\label{ex:ex2-1:expectation_and_variance:solution}
    R_i(x) = \min\left(\left(\sum_{j=i}^{n}x_j - \frac{\beta_i}{2\lambda^{\ast}}-\sigma\right)_{+},x_i\right),
\end{equation}
see Appendix~\ref{appa} for details. \hfill $\diamond$ 
\end{example}
\begin{remark}\normalfont
    In the original problem considered by de Finetti, claims are independent
    and only quota-share contracts are considered for each subportfolio, i.e. with $R_i(x)=a_ix_i$. The optimal proportions are then determined as 
    \[
    a_i = \left(1 -\frac{\beta_i\bbE[X_i]}{2\lambda_{\mathrm{Fin}}\Var(X_i)}\right)_{+}
    \]
    with $\lambda_{\mathrm{Fin}} = \frac{1}{4c}\sum_{i=1}^{n}\frac{(\beta_i \bbE[X_i])^2}{\Var(X_i)}$, cf.\ \cite{definetti1940}. Observe that for some values of the $\beta_i$'s, the optimal proportions are then zero, which implies no reinsurance for that subportfolio. In contrast, the overall optimal solution \eqref{ex:ex2-1:expectation_and_variance:solution} of this problem (beyond the restriction to proportional treaties) leads to reinsurance for all subportfolios regardless of the size of risk loading (if the random variables are not almost surely bounded). As a numerical illustration, consider the case where $X_1$ has a $\Gamma(1/2,1/2)$ distribution and $X_2$ a (shifted) Pareto distribution with p.d.f. given by
    \[
    f_{X_2}(x)=324\left(x+3\right)^{-5}, \quad x\geq 0,
    \]
    so that $\bbE[X_1]=\bbE[X_2]=1$ and $\Var(X_1)=\Var(X_2)=2$. Let $\beta_1=0.1$ and $\beta_2=0.25$ (reflecting that relative risk loadings are typically higher for heavy-tailed risks). Assume that the first-line insurer would like to halve the retained total  variance, so that the bound on the retained variance is given by $c=2$. The optimal parameters for \eqref{ex:ex2-1:expectation_and_variance:solution} are $\sigma = 1.8026351$ and $\lambda^{\ast} = 0.0443408$, while the optimal proportions from de Finetti's solution are $a_1=0.6286093$ and $a_2=0.0715233$. Letting $\eta_{\mathrm{Fin}}$ denote the joint distribution implied by de Finetti's solution, we therefore obtain $\calP(\eta_{\mathrm{Fin}})= 0.0807417$, while $\calP(\etastar)= 0.0232948$ for the overall optimal contract \eqref{ex:ex2-1:expectation_and_variance:solution}. Observe that while this represents an improvement of $71.14\%$ of the objective function \eqref{thatone}, the overall expected loss for the cedent under $\eta_{\mathrm{Fin}}$ is $2.0807417$, while under $\etastar$ it is $2.0232948$, so the improvement is still visible, but considerably smaller.  \hfill $\diamond$ 
\end{remark}

\begin{example}\label{ex:ex2-3:VaR_and_expectation}\normalfont
Consider a variant of Example~\ref{ex:ex2-1:expectation_and_variance}, where instead of fixing the variance, the cedent has a constraint on the Value-at-Risk (at some level $\alpha$) of the total retained amount. Assume that $X$ has a density and that the bound on the VaR $c$ satisfies $0< c < \VaR_\alpha\left(\sum_{i=1}^{n} X_i\right)$ to avoid the optimal contract being the one given by full or no reinsurance. {\color{black} The risk measure $\calP$ is thus the same as before, while the functional $\calG$ can be taken to be
\begin{equation*}
    \calG(\eta) = \widehat{\VaR}_{\alpha}({T_{S}}_{\#}\eta) - c,
\end{equation*}
In this case Proposition~\ref{prop:suppport_minimum} is not applicable and one has to resort to the (weaker) Proposition~\ref{prop:support_minimizer}, with set $\calC$ taken as
\begin{equation}
\calC=\{\eta\in\scrM\mid \widehat{\VaR}_{\alpha}({T_{S}}_{\#}\eta)=v^{\ast}\}.
\end{equation}
For $v^{\ast}$ the $ \widehat{\VaR}_{\alpha}$ of an optimal reinsurance contract. By carefully examining the minima of the associated functions, we get that that there exist $d\geq 0$ such that the optimal reinsurance contract is deterministic and given by
\begin{equation}\label{ex:ex2-3:VaR_and_expectation:solution2}
    R(x) =  \left(x_1,\ldots,x_{i-1},\sum_{j=i}^{n}x_j-v^{\ast},0,\ldots,0\right)
\end{equation}
if $\sum_{j=1}^{i-1}\beta_jx_j+\beta_i\sum_{j=i}^{n}x_j-\beta_iv^{\ast}\leq d$, $\sum_{j=i+1}^{n}x_j\leq v^{\ast}$ and $\sum_{j=i}^{n}x_j> v^{\ast}$, or $R(x)=0$ otherwise. See Appendix~\ref{appb} for details.} \hfill $\diamond$ 
\end{example}

A couple of remarks are in order. 
\begin{remark}\label{rem:size_C}\normalfont
The choice of $\calC$ as in the previous example is motivated from the similar result found in \cite{guerra2012quantile} and is convenient because it allows us to get rid of the terms coming from $\calG$. Observe, however, that other sensible choices could have been
\[
\calC'=\{\eta\in \calD\mid \widehat{\VaR}_{\alpha}({T_{S}}_{\#}\eta)\leq v^{\ast}\}
\]
or
\[
\calC''=\{\eta\in \calD\mid \widehat{\VaR}_{\alpha}({T_{S}}_{\#}\eta)\geq v^{\ast}\},
\]
where $\calD$ is the set appearing in the statement of Proposition~\ref{prop:kkt}. While using these sets one might obtain deeper information from the measures $\eta_{x,y,t,\varepsilon}$, one is also left with the task of fully describing $\calC'$ and $\calC''$ and, in particular, proving that they are convex and satisfy \eqref{eq:convex_gateaux}. Hence, when choosing $\calC$, one needs to consider the trade-off between letting it be too big and the simplicity by which one can describe it (as otherwise one could simply take the largest convex set containing $\etastar$ and satisfying \eqref{eq:convex_gateaux}). \hfill $\diamond$ 
\end{remark}
\begin{remark}\normalfont
Observe that for $n=1$, the result in Example~\ref{ex:ex2-3:VaR_and_expectation} agrees with the result in Corollary 1 of \cite{guerra2012quantile} and the current framework might help to explain the resemblance of these results to the one in \cite{wang2005optimal}: Equation \eqref{eq:positive_lagrange} implies that Value-at-Risk constraints might be recast as an optimization of a Lagrangian, so both approaches lead to the same sort of optimal contracts. \hfill $\diamond$ 
\end{remark}
\begin{remark}\normalfont
Notice that in both examples above, the solution is not fully specified, but instead is given in terms of some unknown parameters ($\lambda^{\ast}$ and $\sigma$ in Example~\ref{ex:ex2-3:VaR_and_expectation}, and $v^{\ast}$ in Example~\ref{ex:ex2-1:expectation_and_variance}). As mentioned earlier, this is an unavoidable feature of our procedure, which arises from the dependence of $p_{\etastar}$ and $g_{\etastar}$ on $\etastar$. However, through \eqref{ex:ex2-1:expectation_and_variance:solution} or \eqref{ex:ex2-3:VaR_and_expectation:solution2} we can obtain an expression for $\calP(\etastar)$ where the only unknowns are these parameters. Hence, we may instead treat them as variables and optimize over them (ensuring that the constraints are still satisfied), thus obtaining a full description of the optimal contracts. \hfill $\diamond$ 
\end{remark}
One can easily generalize the approach from Example~\ref{ex:ex2-3:VaR_and_expectation} to consider slightly more complex constraints involving two or more levels for the VaR, for example, $\calG:\scrM \to \bbR^{m}$ with $i$-th component given by
    \[
    g_i(\eta) = \widehat{\VaR}_{\alpha_i}({T_{S}}_{\#}\eta) - c_i,
    \]
    with $1>\alpha_1>\cdots>\alpha_m>0$ and $0\leq c_1\leq \cdots\leq c_m$. In this case, one could choose the set $\calC$ as
    \[
    \calC=\{\eta\in \calD\mid \widehat{\VaR}_{\alpha_i}({T_{S}}_{\#}\eta)= v^{\ast}_i, \; i=1,\ldots, m\},
    \]
    where $v^{\ast}_i = \widehat{\VaR}_{\alpha_i}({T_{S}}_{\#}\eta)$ for an optimal reinsurance contract $\etastar$. Alternatively, this set can also be used to minimize weighted combinations of values-at-risk at different levels, i.e., risk measures of the form $\sum_{i=1}^{m}\beta_i\widehat{\VaR}_{\alpha_i}({T_{S}}_{\#}\eta)$ under, say, a budget constraint in the premium. Such an approach is implicitly used in reinsurance practice when trying to fix quantiles of the target distribution of the cedents in the context of regulatory ruin (where different measures apply to different ``degrees" of insolvency, cf.\ \cite[Ch.8]{albrecher2017reinsurance}). 
    
    For $n=1$, one can consider another direction of generalization of Example~\ref{ex:ex2-3:VaR_and_expectation}: let $\rho$ and $\omega$ denote non-decreasing functions with $\rho(0)=\omega(0)=0$ and $\rho(1)=1$. For $\delta>0$ one can consider the risk measure given by
    \[
    \calP(\eta) = \int (1-\delta)(x-y)\;\eta(dx,dy) + \int_{0}^\infty \omega\left({\pi_2}_\#\eta(t,\infty)\right)\, dt + \delta\int_{0}^\infty \rho\left({\pi_2}_\#\eta(t,\infty)\right)\, dt.
    \]
    This risk functional corresponds to minimizing the risk-adjusted liability of the cedent, a scenario considered by \cite{cheung2017characterizations}. The risk incurred by the first-line insurer is measured through a distortion risk measure, which is then represented by the last integral in the definition of $\calP$. We do not study this case here, but note that the techniques seem to extend to this case by letting $\calC$ capture the discontinuities/points of non-differentiability of the functions $\omega$ and $\rho$.

    At this point, one can see that Propositions~\ref{prop:support_minimizer} and \ref{prop:suppport_minimum} can be applied to several situations and are particularly well-suited whenever the risk measure or the constraints can be written by means of (functions of) integrals. As a further example, we mention that the methodology can be applied to deal with more complex situations such as the ones considered in \cite{kaluszka2004mean}. There, one would like to minimize risk measures of the form 
    \[
    \calP(\eta) = f\left(\int_{\bbR^{n}_+}p_1(x-y)\;\eta(dx,dy),\ldots,\int_{\bbR^{n}_+}p_\ell(x-y)\;\eta(dx,dy) \right)
    \]
    subject to the constraints $\calG = (g_1,\ldots,g_m)$ given by 
    \[
    g_i(\eta) = h_i\left(\int_{\bbR^{n}_+}q_{i,1}(y)\;\eta(dx,dy),\ldots,\int_{\bbR^{n}_+}p_{i,\ell_{i}}(y)\;\eta(dx,dy) \right),
    \]
    where all the $p_i$'s and $q_{i,j}$'s are (multivariate) rational functions and $f$ and the $h_i$'s are differentiable.\footnote{The setting in \cite{kaluszka2004mean} allows for equalities in the constraints, and for $f$ and the $h_i$'s to not be differentiable at the expense of being increasing in one variable. Equalities can be handled in the same way as in Examples~\ref{ex:ex2-1:expectation_and_variance} and \ref{ex:ex2-3:VaR_and_expectation}, and while we cannot get rid of the differentiability requirement, all of the examples in \cite{kaluszka2004mean} seem to satisfy this as well.} For some particular choices of functions, one can even immediately see that the solutions are deterministic  by means of the techniques developed in Section~\ref{secOT} (for example, for the case when $\calP$ is given as the variance and the constraints depend only on the second marginal, which corresponds to the situation in Theorem 1 in \cite{kaluszka2004mean}).
    
    Loosely speaking, our results also give an intuitive explanation to the ubiquity of stop-loss contracts in optimal reinsurance problems: often, the function $p_{\etastar}+\lambda^{\ast}\cdot g_{\etastar}$ can be written in the form
    \[
    p_{\etastar}(x,y)+\lambda^{\ast}\cdot g_{\etastar}(x,y) = \hat{p}_{\etastar}(x-y)+\lambda^{\ast}\cdot \hat{g}_{\etastar}(x-y)
    \]
    for some functions $\hat{p}_{\etastar}$ and $\hat{g}_{\etastar}$ such that $\hat{p}_{\etastar}+\lambda^{\ast}\cdot \hat{g}_{\etastar}$ has one minimum. According to Propositions \ref{prop:support_minimizer} and \ref{prop:suppport_minimum}, the optimal contract then needs to satisfy $x-y=c$ for some constant $c$, which together with the condition $(x,y)\in \calA_R$, implies that $y=(x-c)_{+}$, which is the form of a stop loss contract.

Finally, we would like to point out that while the contracts in \eqref{ex:ex2-1:expectation_and_variance:solution} and \eqref{ex:ex2-3:VaR_and_expectation:solution2} are deterministic --- in the sense that knowledge of $X$ implies knowledge of $R(X)$, --- the contracts for the individual subportfolios are still random, since $X_i$ stand-alone is not enough to fully specify $R_i(X)$. This is in contrast to \cite{guerra2021reinsurance}, where it is enforced that, conditional on $X_i$, $R_i(X)$ is independent of the remaining contracts in the portfolio. While going into separate contracts with potentially different reinsurers with such marginally random contracts may be challenging in current reinsurance practice, reinsuring all these subportfolios with the same reinsurer (but possibly different safety loadings $\beta_i$, e.g.\ due to different business lines) may be quite feasible. Compared to alternatives, such a contract just leads to a slightly more involved (but deterministic) formula for settling the overall reinsured amount once all claim data for the considered time period are available. In some sense, a part of the risk diversification is done \textit{in house} this way, which is quite common for certain types of aggregate reinsurance covers in practical use, see e.g.\ \cite{albrecher2017reinsurance}.

The following examples (re)examine some of the classical problems in optimal reinsurance through the lens of optimal transport. We use the results we developed in Section~\ref{secOT}.
\begin{example}\label{ex:ex1}\normalfont
Let $n=1$ and assume that $X$ has finite variance. Let $\calP$ be given by
\begin{equation*}
	\calP(\eta)= \widehat{\Var}(T_{\#}\eta):=\int x^{2}\,T_{\#}\eta(dx) - \left(\int x \, T_{\#}\eta(dx)\right)^{2}
\end{equation*}
and $S=\{\eta\in \scrM\mid \int y\, {\pi_{2}}_{\#}\eta(dy) = c\}$ for some $c\geq 0$. This is the classical example (see e.g.\ \cite{Pesonen1984}) where the objective is to minimize the retained variance of the insurer subject to a fixed reinsurance premium which is computed through the expected value principle. In this case, the optimal reinsurance contract is known to have the deterministic form 
\begin{equation}\label{eq.optVar}
    \eta^{*}=(\mathrm{Id},R_{SL,a^{*}})_\#(\mu)\quad \text{ for some $a^{*}\geq 0$},
\end{equation}
where, for $a\geq 0$, $R_{SL,a}$ is the function on $\bbR$ given by $R_{SL,a}(x)=(x-a)_{+}$, i.e., a stop-loss contract is optimal. Note that finiteness of the variance of $X$ implies that the set $S$ is closed. In order to apply the results from Section~\ref{secOT}, observe that
\begin{equation}\label{prec1}
\begin{split}
\inf_{\eta\in S}\calP(\eta)&= \inf_{\eta\in\scrM, \overline{{\pi_2}_\#\eta}=c}\int_{\bbR\times\bbR}(x-y)^2\eta(dx,dy)-\left(\int_{\bbR\times\bbR}(x-y)\eta(dx,dy)\right)^2 \\
&= \inf_{\eta\in\scrM, \overline{{\pi_2}_\#\eta}=c}\int_{\bbR\times\bbR}(x-y)^2\eta(dx,dy)-\left(\overline{\mu}-c\right)^2\\
&= \inf_{
\substack{\nu\in\calP(\bbR), \\ \overline{\nu}=c, \nu\prec_1\mu}}\inf_{\eta\in\Pi(\mu,\nu)}\int_{\bbR\times\bbR}(x-y)^2\eta(dx,dy)-\left(\overline{\mu}-c\right)^2.
\end{split}
\end{equation}
From \eqref{prec1}, we observe that the conditions of Proposition \ref{prop:monotonic_rea} are satisfied. Hence, it follows that, with
\[
x\mapsto g_\nu(x):=F_\nu^{-1}\circ F_{\mu}(x),
\]
the coupling $\pi_\nu:=(\mathrm{Id},g_\nu)_\#\mu$ is optimal for the {\color{black}inner (classical OT) problem in the last line of \eqref{prec1}}, so we get
\begin{equation}\label{eq.co}
\inf_{\eta\in\scrM, \overline{{\pi_2}_\#\eta}=c}\int_{\bbR\times\bbR}(x-y)^2\eta(dx,dy)=  \inf_{
\substack{\nu\in\calP(\bbR), \\ \overline{\nu}=c, \nu\prec_1\mu}}\int_{\bbR}(x-g_\nu(x))^2\mu(dx).
\end{equation}
It follows that the optimizer is given by the contract $\eta^*$ 
given in \eqref{eq.optVar}, with $R_{SL,a^{*}}=g_{\nu^*}$, for $\nu^*$ minimizer in \eqref{eq.co}.
Indeed, we want to minimize the integral w.r.t.\ $\mu$ of $(x-g_\nu(x))^2$, over functions $g_\nu$ which are non-increasing, below $\mathrm{Id}$, and such that the area below them (i.e.\ the integral w.r.t. $\mu$) is fixed (equal to $c$). Then clearly the optimal $g_{\nu^*}$ is parallel to $\mathrm{Id}$, thus of the form $(x-a^*)_+$, with $a^*$ determined by the constraint $\mathbb{E}[(X-a^*)_+]=\int y\, ({g_{\nu^*}}_\#\mu)(dy)=\int y\, \nu^*(dy) = c$. Hence, the OT approach provides an alternative proof of this classical result. \hfill $\diamond$ 
\end{example}

\begin{example}\label{ex:ex2}\normalfont
If we modify the previous example by setting instead 
\[
	S=\{\eta\in \scrM\mid \widehat{\Var}({\pi_{2}}_{\#}\eta) = c\}
\]
for some $c\geq 0$, we obtain the situation where the objective is still to minimize the retained variance, but now subject to a fixed reinsurance premium loading that is proportional to the variance (cf.\ \cite{Pesonen1984}). In this case, the optimal reinsurance contract is known to be deterministic and of the form $\eta^{*}=(\mathrm{Id},R_{QS,a^{*}})_\#(\mu)$ for some $0\leq a^{*}\leq 1$, where $R_{QS,a}(x)=ax$, i.e., a so-called quota-share contract is optimal. Let $\calP$ denote the same functional as in the previous case, and observe that it still satisfies Assumptions \ref{assump:A1} and \ref{assump:A2} above with function $p_{\eta}$ given by
\[
p_{\eta}(x,y) = \left(x-y\right)^{2}-2\overline{{T}_{\#}\eta}(x-y).
\]
This function and $S$ satisfy the conditions from Proposition~\ref{prop:monotonic_rea}, so it follows that, for $\nu\in \pi_2(\calS)$, the minimum of \eqref{eq:minimum_etastar_OT} is achieved through couplings $\pi_\nu$ of the form $\pi_\nu:=(\mathrm{Id},g_\nu)_\#\mu$ with
\[
x\mapsto g_\nu(x):=F_\nu^{-1}\circ F_{\mu}(x).
\]
Plugging this coupling into the definition of $\calP$, we obtain
\begin{align*}
    \calP(\pi_\nu) &= \int_{0}^{\infty}\left(x-g_\nu(x)\right)^{2}\;\mu(dx) - \left(\int_{0}^{\infty}(x-g_\nu(x)) \; \mu(dx)\right)^{2}\\
    &=\int_{0}^{1}\left(F_{\mu}^{-1}(x)-F_\nu^{-1}(x)\right)^{2}\;dx - \left(\int_{0}^{1}F_{\mu}^{-1}(x)-F_\nu^{-1}(x) \; dx\right)^{2}
\end{align*}
and the problem is reduced to finding the optimal distribution function $F_\nu$. Deferring the computations to Appendix \ref{appc}, we find that the optimal reinsurance contract is given by
\[
R(X) = \frac{X-a}{(1-\lambda)},
\]
for  $\lambda = 1 - \sqrt{\Var(X)/c}$ and any $0 \leq a\leq F_{\mu}^{-1}(0)$. Choosing $a = 0$, we obtain the quota-share contract known to be optimal \cite{Pesonen1984} (and any other choice of $a$ would just lead to a deterministic (``side'') payment from the reinsurer to the insurer (see e.g.\ \cite{gerberpafumi}), which would be priced in the reinsurance premium in an additive way, as its variance is zero, and so would only lead to a deterministic additional exchange and serve no purpose). Again, the OT approach in this way provides an alternative proof of this classical result. 

Observe that in this situation one cannot directly apply Proposition~\ref{prop:factorization_marginals}: given an arbitrary reinsurance treaty $\eta$, the monotonic rearrangement between $\mu$ and $\nu=T(\eta)$ leads to a function $R$ such that $R(X)$ is distributed according to $\nu$. However, it is in general not the case that $X-R(X)$ is distributed according to $\pi_{2}(\eta)$, so we cannot guarantee that $\Var(X-R(X))=c$. \hfill $\diamond$ 
\end{example}
For the next example, we take the viewpoint of the reinsurer in the optimization problem. It will lead to a situation where introducing \emph{external randomness} is indeed optimal. 
\begin{example}\label{ex:ex6}\normalfont
For simplicity of exposition, here the second marginal of a reinsurance treaty will refer to the deductible rather than the reinsured amount.
\newpage
\noindent\textit{{\color{black}General formulation}}\\[0.2cm]
For each $k=1,\ldots, n$, let $\nu_{k}\in \scrP(\bbR^{+})$ denote a predefined distribution. For any lower semi-continuous $\calP$, we can set 
\[
S=\{\eta \in \scrM: {\pi_{2,k}}_{\#}\eta=\nu_{k},\; k=1,\ldots,n\}.
\]
One can interpret this situation as follows: $n$ insurers ask each for a target distribution $\nu_{k}$, $k=1,\ldots,n$, after reinsurance, and the reinsurer tries to minimize $\calP$ respecting these target distributions (potentially involving the introduction of randomized treaties). Phrased in optimal transport terms, this example corresponds to a problem in the area of multi-marginal optimal transport, a generalization from the classical transport problem in which there might be more than one target measure. While we do not attempt to solve the problem in general, we point out some of the insights obtained from seeing the problem from this perspective and solve it for one particular case. Consider the case $n=2$, $\mu$ absolutely continuous with finite second moments and $\calP$ given as the variance of the sum of the reinsured amounts. For any $\eta\in S$, we then have
\begin{align*}
    \calP(\eta) &= \int_{\bbR^{4}_{+}}(x_1-y_1+x_2-y_2)^2\;\eta(dx_1,dx_2,dy_1,dy_2) - \left(\overline{\mu_1}-\overline{\nu_1}+\overline{\mu_2}-\overline{\nu_2}\right)^2.
\end{align*}
The second term on the right-hand side is fixed. Hence, the problem is equivalent to minimizing the functional 
\begin{equation}\label{eq:ex6:q}
\calQ(\eta) =\int_{\bbR^{4}_{+}}(x_1-y_1+x_2-y_2)^2\;\eta(dx_1,dx_2,dy_1,dy_2) 
\end{equation}
on $\Pi(\mu,\nu_1,\nu_2)\cap \scrM$, the set of couplings between $\mu$, $\nu_{1}$ and $\nu_{2}$ supported on $\calA_{R}$. Since $X_{1}$ and $X_{2}$ have finite variance, this minimization problem is finite. As can be seen from quick inspection, the task of solving the OT problem with the reinsurance restrictions is challenging and for arbitrary $\mu$, there is no guarantee that there even exist solutions that can be expressed in terms of elementary terms so that numerical solutions have to be considered. Nevertheless, we provide an illustration of how this can lead to explicit results in specific cases. \\[0.2cm]
\textit{{\color{black}Numerical solution: Setting}}\\[0.2cm]
Assume $X_1$ and $X_2$ are independent, $X_1$ has a lognormal distribution with p.d.f. given by
\[
f_{X_1}(x)=\frac{1}{x\sqrt{2\pi\log(3)}}\exp\left(-\frac{(\log (\sqrt{3}x) )^2}{2\log(3)}\right), \quad x>0,
\]
and $X_2$ a (shifted) Pareto distribution with p.d.f given by
\[
f_{X_2}(x)=324\left(x+3\right)^{-5}, \quad x\geq 0.
\]
Here the parameters are chosen such that $\bbE[X_1]=\bbE[X_2] = 1$ and $\Var(X_1)=\Var(X_2)=2$. As there is no standard solution method for the multi-marginal transportation problem with arbitrary cost, we utilize a discretized setting (see e.g.\ \cite{peyre2019computational} and Appendix \ref{appd} for details). For $N\geq 1$, let $\tilde{X}_1$ and $\tilde{X}_2$ be discretized versions of $X_1$ and $X_2$ obtained after binning the $X_i$'s into $N$ bins of equal length up to a (high) quantile, assigning probabilities according to their distributions and putting the remaining mass into the last bin to account for the unboundedness of the distributions. Let $\tilde{Y}_1$ and $\tilde{Y}_2$ be two random variables such that
\[
\tilde{Y}_1 \overset{d}{=} 0.5\tilde{X}_1, \quad \tilde{Y}_2 \overset{d}{=} \min(\tilde{X}_2,0.5)+0.25(\tilde{X}_2-0.95)_+,
\]
where $\overset{d}{=}$ denotes equality in distribution. The idea behind this choice is that $\tilde{Y}_1$ could arise from applying a quota-share contract (with proportionality factor 0.5) to $\tilde{X}_1$, while $\tilde{Y}_2$ could be the retained amount from a bounded stop-loss contract on $\tilde{X}_2$ (with deductible 0.5 and layer size 0.45), where the reinsurer still takes 75\% of the exceedance above that layer. The distributions $\tilde{\nu}_1$ and $\tilde{\nu}_2$ of $\tilde{Y}_1$ and $\tilde{Y}_2$ respectively are the target distributions of the two insurers, and the task is now to see how the reinsurer can offer these while keeping the variance of $(\tilde{X}_1-\tilde{Y}_1)+(\tilde{X}_2-\tilde{Y}_2)$ (the reinsured amount) minimal, for instance in order to provide a competitive reinsurance premium. Let $\tilde{\mu}$ denote the joint distribution of $(\tilde{X}_{1},\tilde{X}_{2})$.\\[0.2cm]
\textit{{\color{black}Numerical solution: Results}}\\[0.2cm]
The problem corresponds to the multimarginal transport problem of moving mass from $\tilde{\mu}$ to $\tilde{\nu}_1$ and $\tilde{\nu}_2$ with cost \eqref{eq:ex6:q}. Although potentially high-dimensional, this is a relatively easy linear optimization exercise, which for $N=40$, we solve by using standard linear optimization packages. The results for some of the bivariate distributions of $(\tilde{X}_1,\tilde{X}_2,\tilde{Y}_1,\tilde{Y}_2)$ are shown in Figures~\ref{fig:pmfXsYs} and \ref{fig:remaining_pmf}. Let $\eta_{\mathrm{Det}}$ denote the distribution of 
\[
(\tilde{X}_1,\tilde{X}_2,0.5\tilde{X}_1,\min(\tilde{X}_2,0.5)+0.25(\tilde{X}_2-0.95)_+),
\]
$\etastar$ the distribution found by optimization and denote by $\tilde{R}_i$ the reinsured amount $\tilde{X}_i-\tilde{Y}_i$. 
\begin{figure}[H]
  \begin{subfigure}{.5\linewidth}
  	\includegraphics[width=\linewidth]{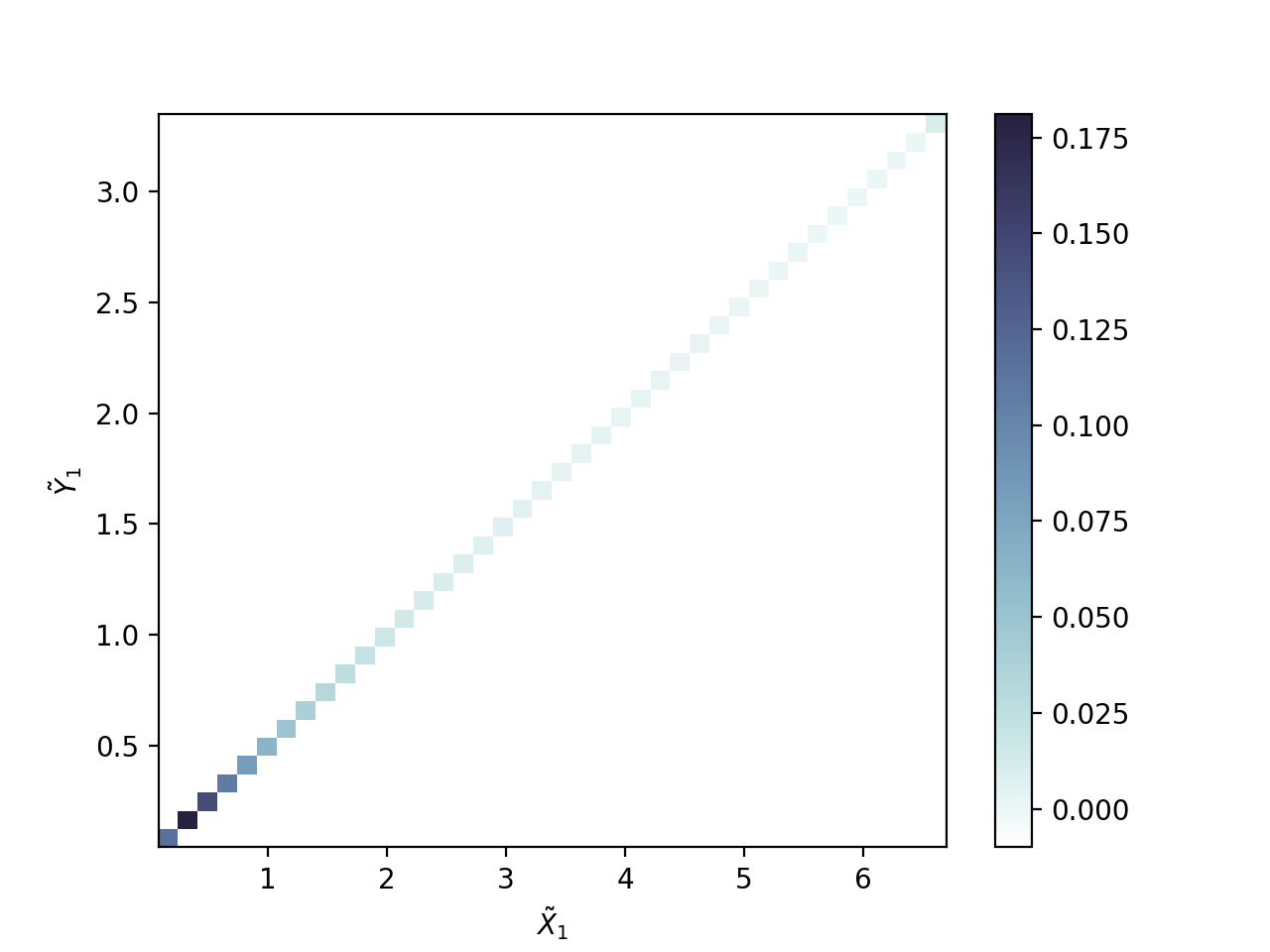}
  	\label{fig:pmfX1Y1:det}
  \end{subfigure}%
  \begin{subfigure}{.5\linewidth}
  	\includegraphics[width=\linewidth]{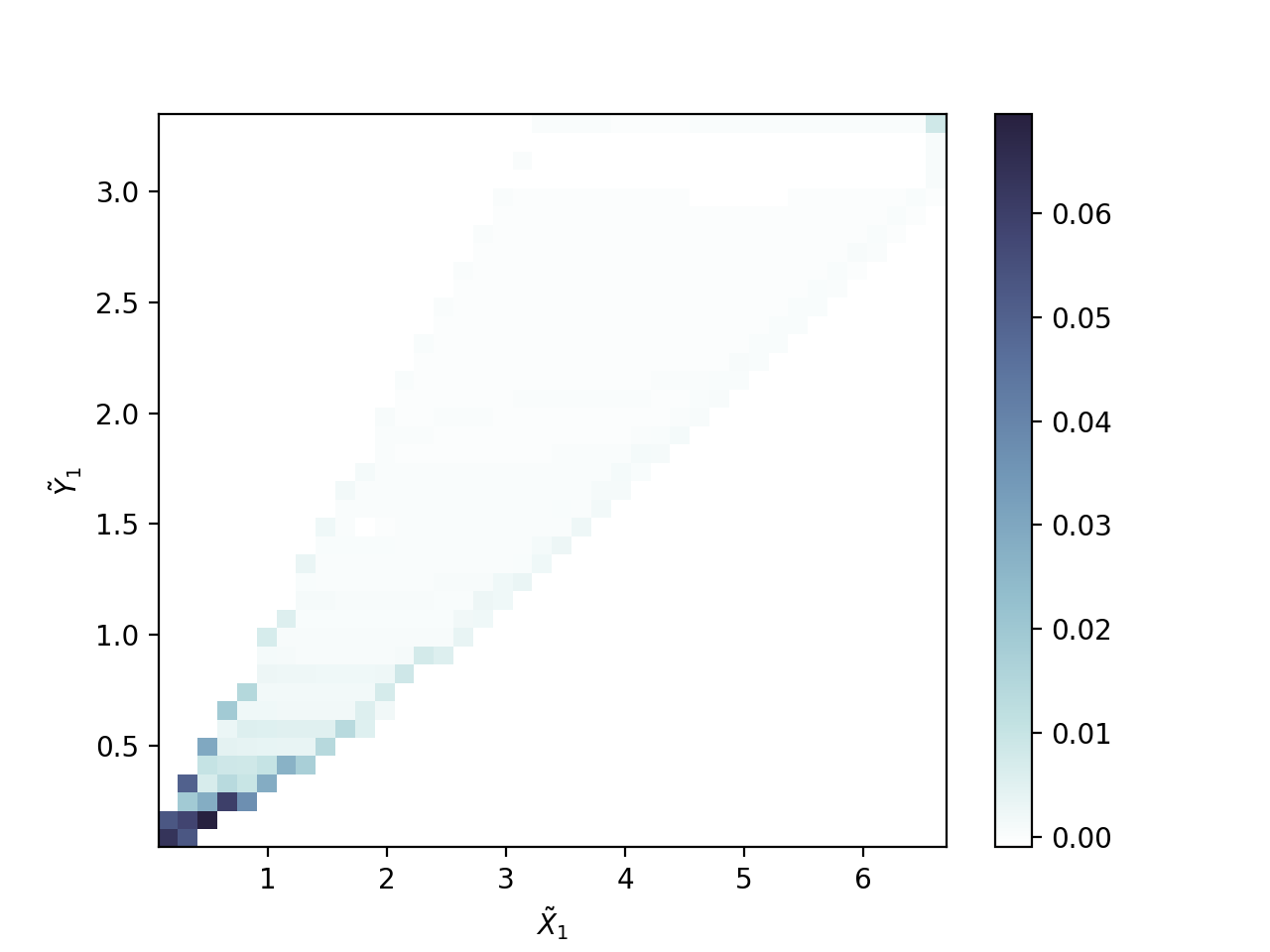}
  	\label{fig:pmfX1Y1:OT}
  \end{subfigure}
  \begin{subfigure}{.5\linewidth}
  	\includegraphics[width=\linewidth]{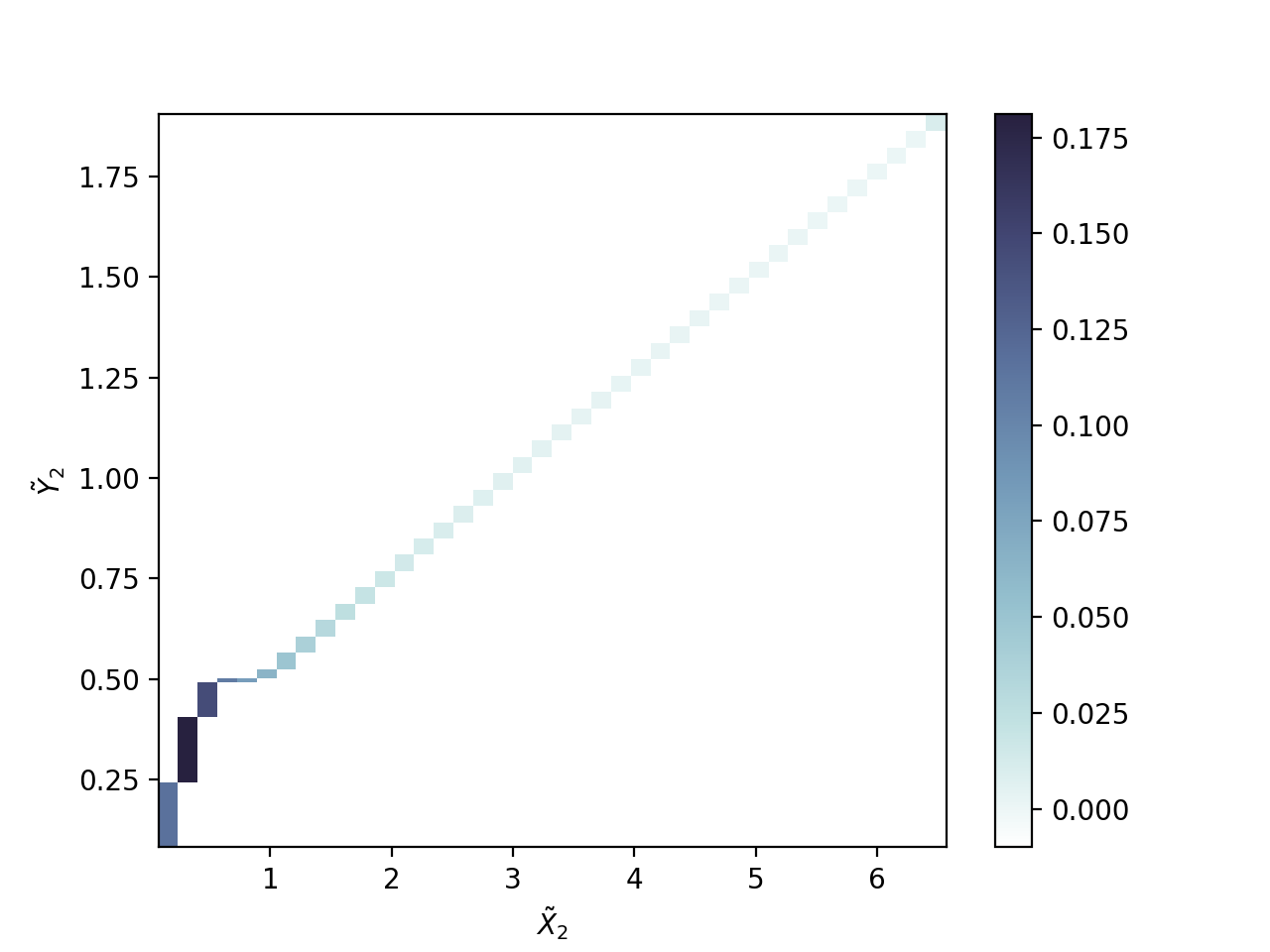}
  	\label{fig:pmfX2Y2:det}
  \end{subfigure}%
  \begin{subfigure}{.5\linewidth}
  	\includegraphics[width=\linewidth]{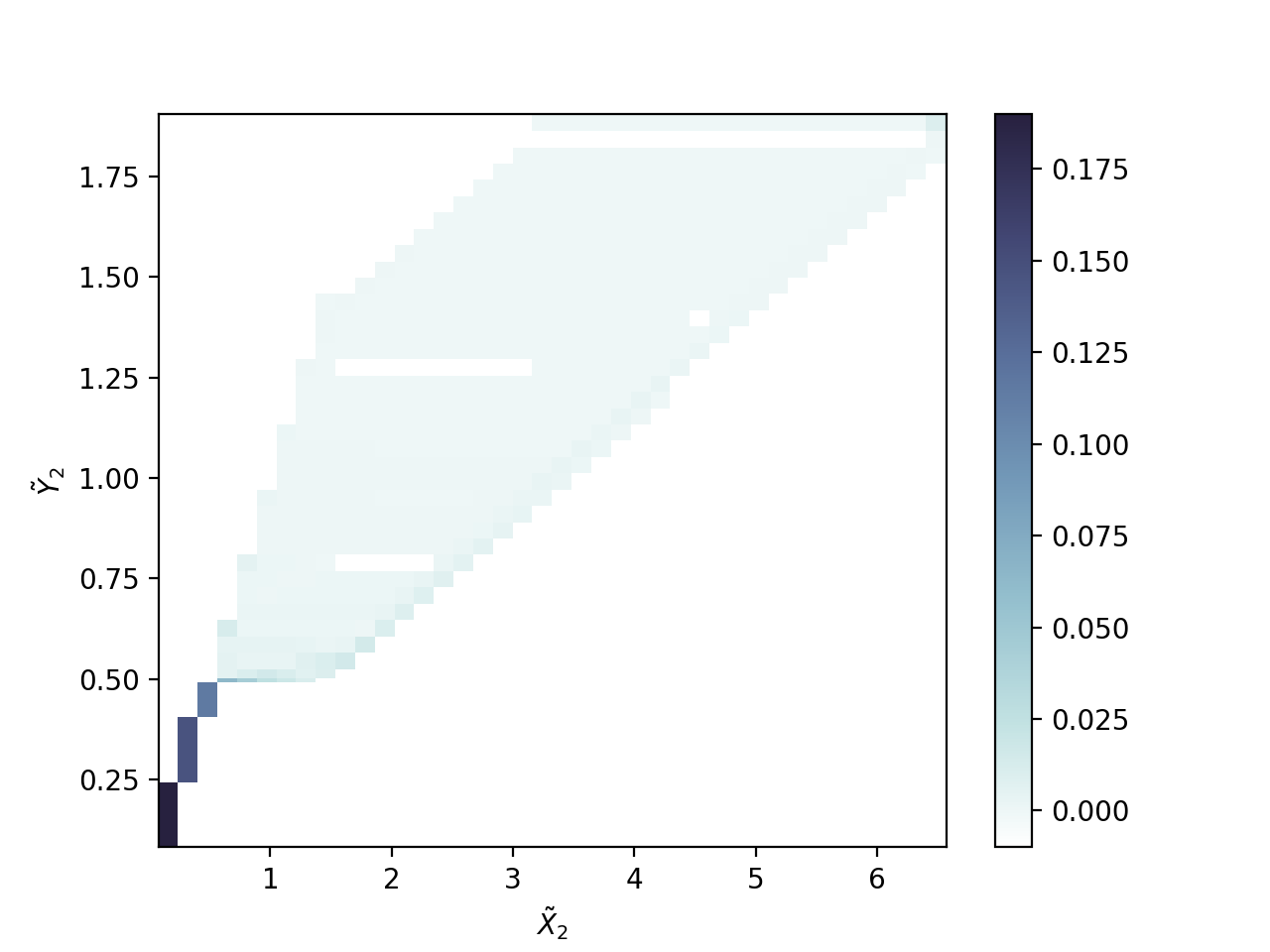}
  	\label{fig:pmfX2Y2:OT}
  \end{subfigure}
  \caption{Probability mass functions for $\tilde{X}_1$ and $\tilde{Y}_1$ (upper row) and $\tilde{X}_2$ and $\tilde{Y}_2$ (lower row) under $\eta_{\mathrm{Det}}$ (left column) and $\etastar$ (right column).}
  \label{fig:pmfXsYs}
\end{figure}
\newpage
With the dependence structure indicated by $\eta_{\mathrm{Det}}$, $\tilde{R}_1$ and $\tilde{R}_2$ are independent and $\Var_{\eta_{\mathrm{Det}}}(\tilde{R}_1+\tilde{R}_2) = 1.05314$. The variance after optimization is $\Var_{\etastar}(\tilde{R}_1+\tilde{R}_2) = 0.82875$, which represents an improvement of $21.31\%$. As can be seen from Figures~\ref{fig:pmfXsYs} and \ref{fig:remaining_pmf}, this is achieved in two ways: first, the joint distributions of $(\tilde{X}_1,\tilde{Y}_1)$ and $(\tilde{X}_2,\tilde{Y}_2)$ are changed in such a way that, under $\etastar$, $\tilde{Y}_1$ has positive probability of being close to $\tilde{X}_1$ and $\tilde{Y}_2$ has positive probability of being close to $\tilde{X}_2$ thus allowing the reinsured amounts $\tilde{R}_1$ and $\tilde{R}_2$ to take smaller values than under $\eta_{\mathrm{Det}}$.
Secondly, the variance is also reduced by introducing a positive dependence relationship between $\tilde{Y}_1$ and $\tilde{Y}_2$. This may be slightly counter-intuitive at first if we think of the variance as being reduced by making $\tilde{R}_1$ and $\tilde{R}_2$ counter-monotonic. While Figure~\ref{fig:remaining_pmf} seems to indicate that $\tilde{R}_1$ and $\tilde{R}_2$ have a negative dependence structure, we cannot fully expect it to be counter-monotonic. Under the assumption of counter-monotonicity, small values for $\tilde{R}_1$ would be coupled with larger values of $\tilde{R}_2$, which would imply that values of $\tilde{Y}_1$ close to $\tilde{X}_1$ would be paired with values of $\tilde{Y}_2$ far away from $\tilde{X}_2$. Since $\tilde{Y}_1$ and $\tilde{Y}_2$ are bounded by $\tilde{X}_1$ and $\tilde{X}_2$, this would imply that small values for $\tilde{Y}_1$ would be coupled with larger values of $\tilde{Y}_2$. However, this argument lacks to take into account the fact that the reduction in variance can also be achieved by introducing a different dependence relationship between $\tilde{X}_1$ and $\tilde{Y}_2$, and between $\tilde{X}_2$ and $\tilde{Y}_1$.
\begin{figure}[H]
  \begin{subfigure}{.5\linewidth}
  	\includegraphics[width=\linewidth]{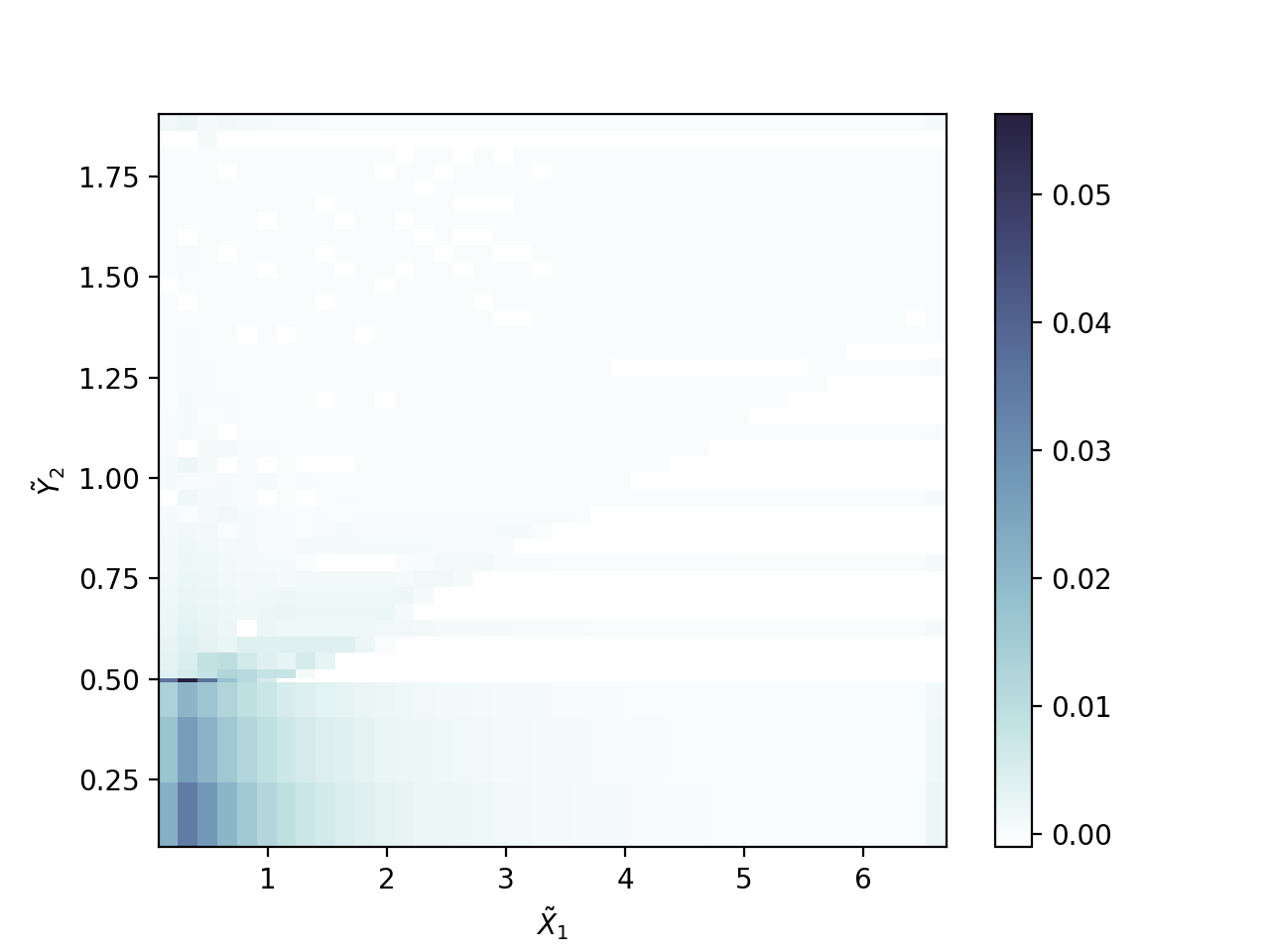}
  	\label{fig:pmfX1Y2:OT}
  \end{subfigure}%
  \begin{subfigure}{.5\linewidth}
  	\includegraphics[width=\linewidth]{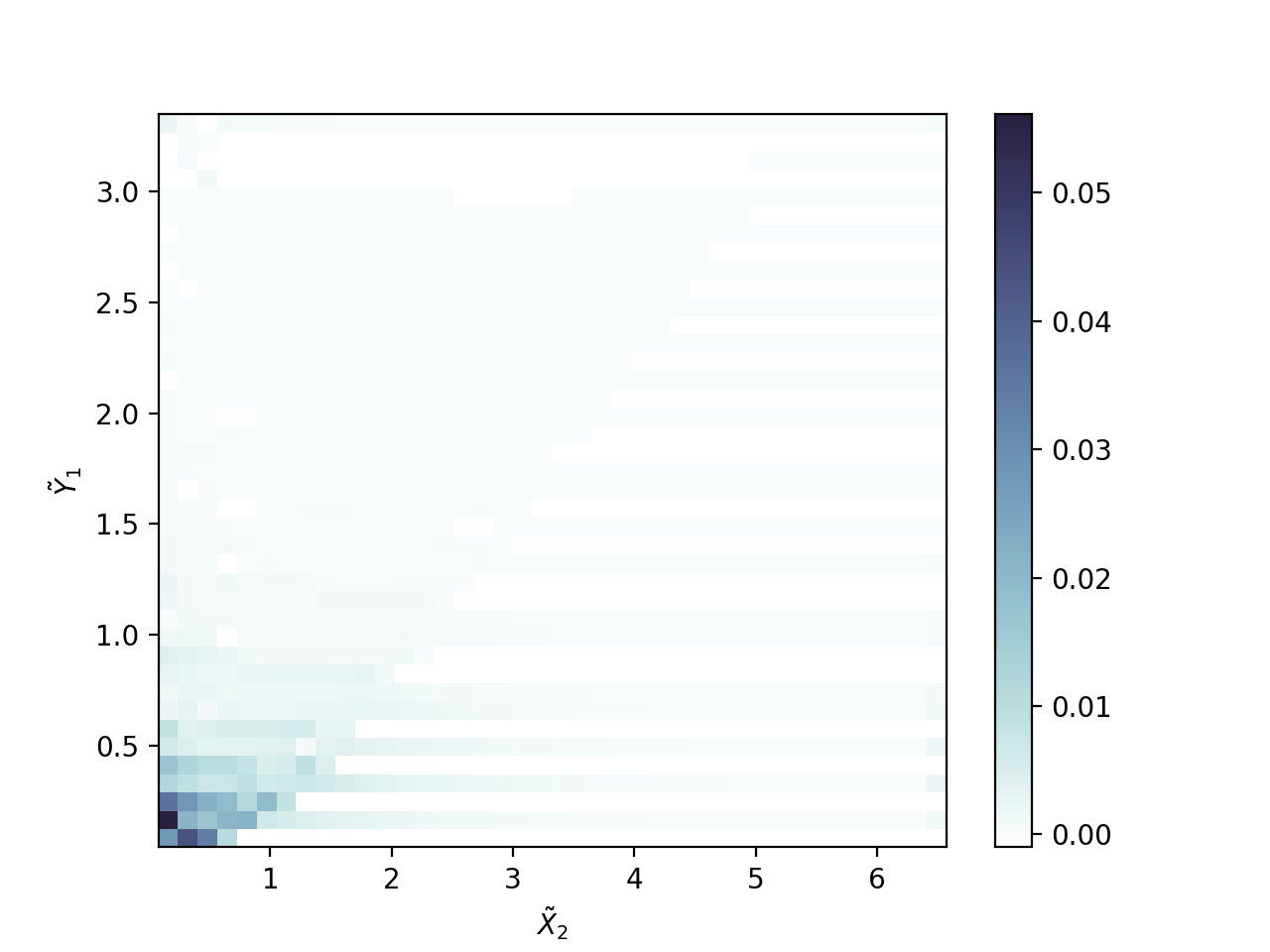}
  	\label{fig:pmfX2Y1:OT}
  \end{subfigure}
  \begin{subfigure}{.5\linewidth}
  	\includegraphics[width=\linewidth]{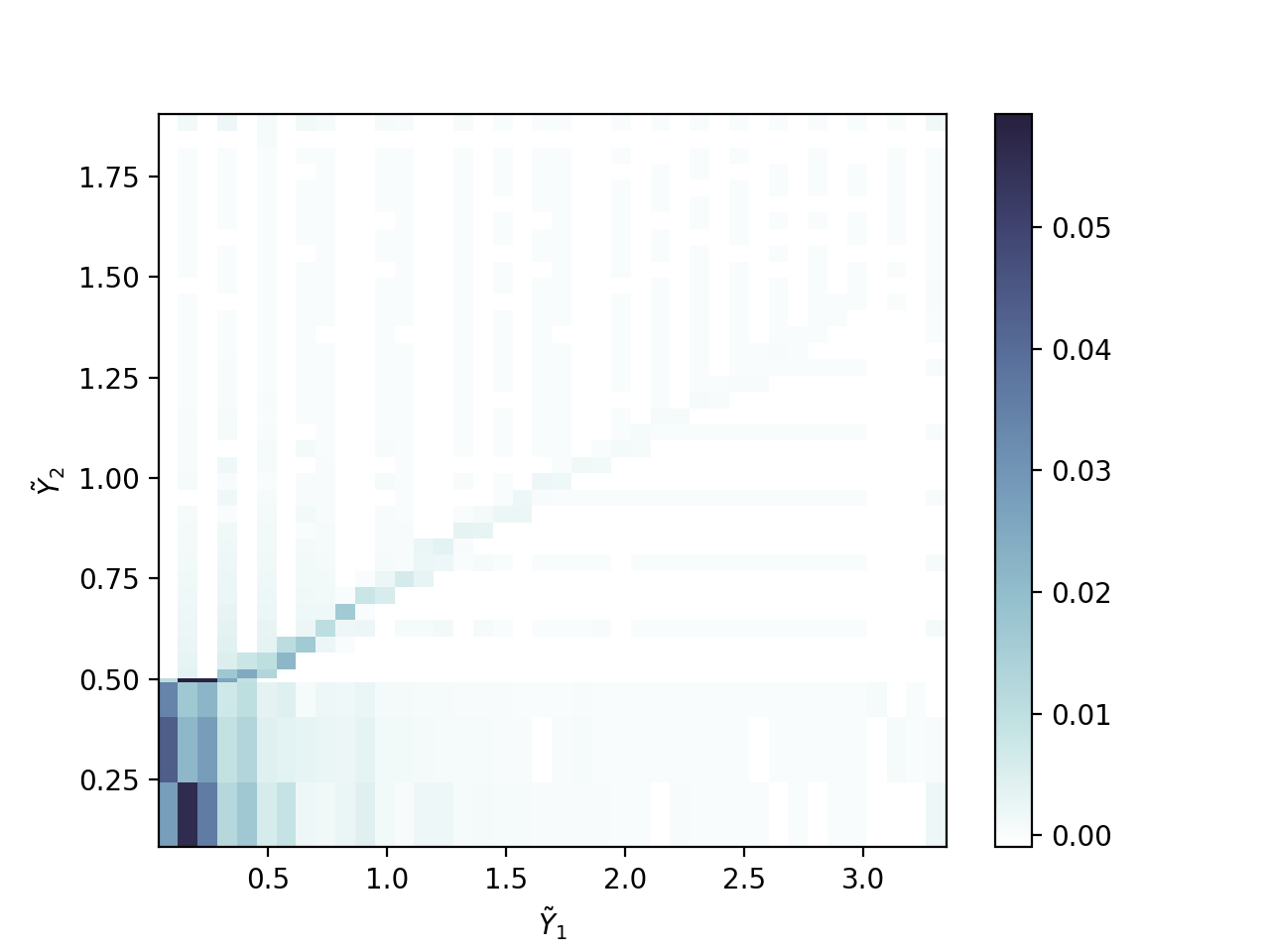}
  	\label{fig:pmfY1Y2:OT}
  \end{subfigure}%
  \begin{subfigure}{.5\linewidth}
  	\includegraphics[width=\linewidth]{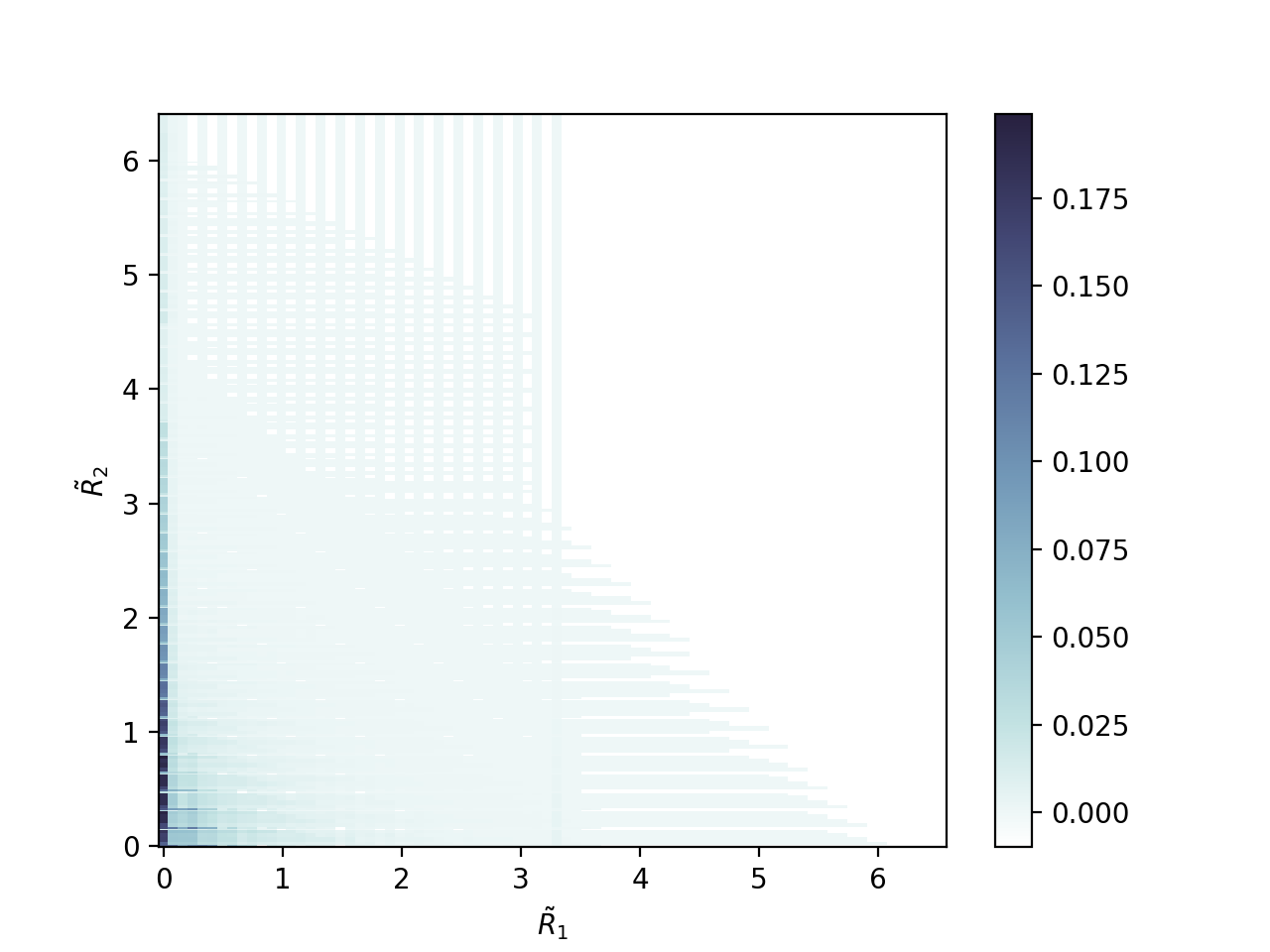}
  	\label{fig:pmfR1R2:OT}
  \end{subfigure}
  \caption{Probability mass functions for the joint distributions under $\etastar$ of $(\tilde{X}_1,\tilde{Y}_2)$ (top left), $(\tilde{X}_2,\tilde{Y}_1)$ (top right), $(\tilde{Y}_1,\tilde{Y}_2)$ (bottom left) and $(\tilde{R}_1,\tilde{R}_2)$ (bottom right).}
  \label{fig:remaining_pmf}
\end{figure}

\begin{figure}[H]
  \begin{subfigure}{.5\linewidth}
  	\includegraphics[width=\linewidth]{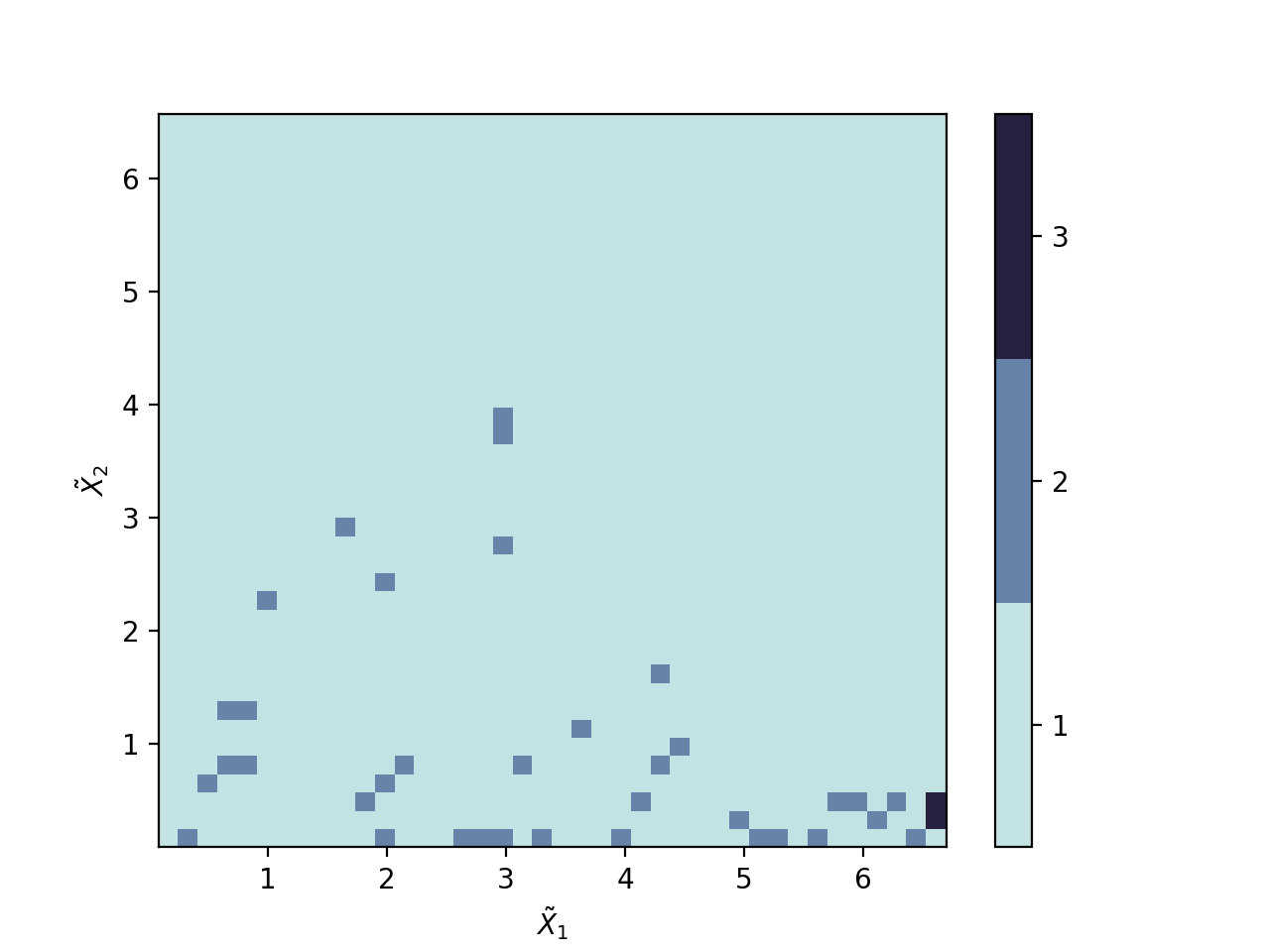}
  	\label{fig:suppX1Y2Y1:OT}
  \end{subfigure}%
  \begin{subfigure}{.5\linewidth}
  	\includegraphics[width=\linewidth]{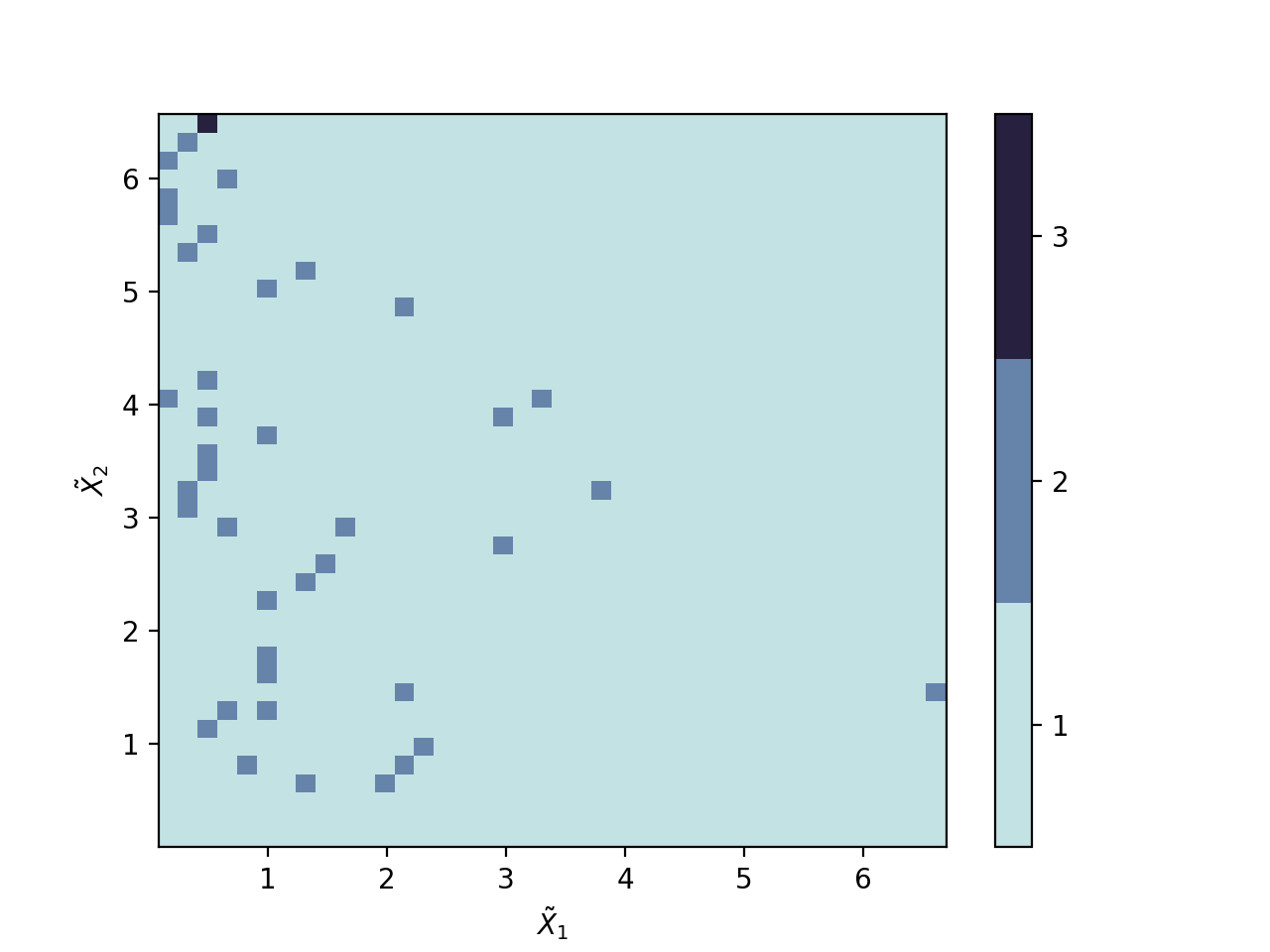}
  	\label{fig:suppX1Y2Y2:OT}
  \end{subfigure}
  \caption{Cardinality of the supports of the conditional distributions of $\tilde{Y}_1$ (left) and $\tilde{Y}_2$ (right) under $\etastar$ given $\tilde{X}_1$ and $\tilde{X}_2$.}
  \label{fig:suppX1Y2Ys}
\end{figure}
Curiously enough, observe that the optimal joint distributions of $(\tilde{X}_1,\tilde{Y}_2)$ and $(\tilde{X}_2,\tilde{Y}_1)$ have a tendency to be concentrated in the upper left corner (as well as in the area where $\tilde{Y}_2\leq 0.5$ for $(\tilde{X}_1,\tilde{Y}_2)$). This seems to indicate that $\tilde{Y}_2$ ``uses'' the extra degree of freedom in the distribution $(\tilde{X}_1,\tilde{Y}_2)$ to compensate for the behavior when $\tilde{X}_2\leq 0.5$ (in which we necessarily must have $\tilde{Y}_2=\tilde{X}_2$, regardless of the joint distribution) and for the constraint $\tilde{Y}_2\leq \tilde{X}_2$. Similarly, $\tilde{Y}_1$ ``uses'' the extra degree of freedom in the distribution $(\tilde{X}_2,\tilde{Y}_1)$ to compensate for the constraint $\tilde{Y}_1\leq \tilde{X}_1$. Finally, from Figure \ref{fig:pmfXsYs} we can observe that, under $\etastar$, there is not a deterministic association between $\tilde{X}_1$ and $\tilde{Y}_1$ nor between $\tilde{X}_2$ and $\tilde{Y}_2$. 

In contrast to Examples~\ref{ex:ex2-1:expectation_and_variance} and \ref{ex:ex2-3:VaR_and_expectation}, the randomization of $\tilde{Y}_1$ for given $\tilde{X}_1$ is not only due to the realization of $\tilde{X}_2$ (and vice versa), as the conditional distributions of $\tilde{Y}_1$ and $\tilde{Y}_2$ given $(\tilde{X}_1,\tilde{X}_2)$ are not concentrated in one point, cf.\ Figure~\ref{fig:suppX1Y2Ys}.
That is, for minimizing the variance, in this example one uses a degree of randomness external to $(\tilde{X}_1,\tilde{X}_2)$ (like for instance a lottery). This external randomness helps, in a way, to reach rather intuitive joint distributions of the variables involved. {\color{black} Observe that, while this does not discard deterministic solutions, this shows that external randomness can also result in optimal contracts.} \hfill $\diamond$ 
\end{example}

\section{Conclusion and Outlook}\label{secconc}
In this paper we provided a link between the two fields of optimal transport and optimal reinsurance, which allows for a reinterpretation of some classical optimal reinsurance results, a characterization of conditions for the optimality of deterministic treaties as well as the derivation of new results, extending some previous approaches in the literature. In a number of concrete examples we illustrated the benefits of this additional perspective on optimal reinsurance problems. We also established an example with two insurers and a reinsurer where external randomness in the contract specification effectively increases the efficiency.\\

While we worked out several concrete cases in detail, there are a number of directions that could be interesting for future research. In particular, we like to mention the following here. In this paper, we did not explicitly deal with the particular case of convex risk measures, an important class of risk measures on which a large part of the literature is focused (cf.\ \cite{balbas2022risk,cheung2014optimal}). When the risk measure is given, for example, by
\begin{equation*}
	\calP(\eta)=\int u\left(w-x-\int y\,{\pi_{2}}_{\#}\eta)(dy) \right)T_{\#}\eta(dx)
\end{equation*}
for $u:\bbR\to\bbR$ a convex and non-decreasing function, then $\calP$ is convex and its optimization can be addressed by means of Propositions~\ref{prop:support_minimizer}, \ref{prop:suppport_minimum} or \ref{prop:monotonic_rea}, depending on the nature of the constraints. 
Also, for general convex risk measures, it seems that the dual representation is tightly connected with the existence of the directional derivatives at optimal reinsurance contracts, and it will be interesting to connect the present approach with concepts from duality theory.
\appendix
\section{{\color{black}Details of examples}}\label{appa}
{\color{black}Here we present detailed computations for the examples presented in Section \ref{sec5}. As some of the examples require lengthy computations, we have split their proofs into steps.
\subsection{Continuation of Example~\ref{ex:ex2-1:expectation_and_variance}} Recall the definition of the functionals $\calP$ and $\calG$: taking expectations in the equation,
\[
\sum_{i=1}^{n} (X_i-R_i +(1+\beta_i)\bbE[R_i]),
\]
we see that the risk measure can be chosen as
\begin{equation*}
\calP(\eta)=\int\sum_{i=1}^{n}\beta_{i}y_{i} \; \eta(dx,dy).
\end{equation*}
In our notation, $\calG$ can then be written as
\[
\calG(\eta)= \int\left(\sum_{i=1}^{n}(x_{i}-y_{i})\right)^{2}-\left(\sum_{i=1}^{n}\int(x_{i}-y_{i})\;\eta(dx,dy)\right)^{2}\; \eta(dx,dy) - c.
\] 
\vspace*{0.05cm}

\noindent\textit{Step 1: Both $\calP$ and $\calG$ satisfy the assumptions of Propositions~\ref{prop:support_minimizer} and \ref{prop:suppport_minimum} with $\lambda^\ast>0$.}\\[0.2cm]
Letting $\etastar$ be any optimal contract, we observe that the functionals $\calP$ and $\calG$ satisfy the conditions of  with functions $p_{\etastar},g_{\etastar}:\bbR^{n}_{+}\times \bbR^{n}_{+}\to \bbR$ given by
\begin{align*}
    p_{\etastar}(x,y) &= \sum_{i=1}^{n}\beta_{i}y_{i},\\
    g_{\etastar}(x,y) &= \left(\sum_{i=1}^{n}(x_{i}-y_{i})\right)^{2}-2\sigma\sum_{i=1}^{n}(x_{i}-y_{i}),
\end{align*}
with $\sigma = \sum_{i=1}^{n}\overline{(\pi_i\circ T)_{\#}\etastar}$. Observe further that if $\eta$ is the contract that specifies full reinsurance, then
\[
\calG(\etastar) +d\calG(\etastar;\eta-\etastar)= -2c <0,
\]
so it follows that there exists $\lambda^{\ast}\geq 0$ such that the support of $\etastar$ is contained in the minima of $h(x,\cdot) = p_{\etastar}(x,\cdot)+\lambda^{\ast} g_{\etastar}(x,\cdot)$. Note, however, that we cannot have $\lambda^{\ast}=0$, since the minimum of $p(x,\cdot)$ occurs at $y=0$ and this would imply that $\etastar$ is the contract for which no reinsurance takes place, violating the condition $\calG(\etastar)<0$. Recall that by Proposition \ref{prop:suppport_minimum}, we have to identify the minimizers of $h(x,\cdot)$ to describe the support of $\etastar$, so the rest of the proof focuses on finding this support.\\[0.2cm]
\textit{Step 2: The minimizers of $h(x,\cdot)$ agree with the minimizers of $a(x,\cdot)$ where
\[
a(x,y) = -\sum_{i=1}^{n}(\beta_{n}-\beta_{i})y_{i} +\lambda^{\ast}\left(s_Y-s_X+\sigma+\frac{\beta_{n}}{2\lambda^{\ast}}\right)^{2},
\]
with $s_X = \sum_{i=1}^{n}x_i$ and $s_Y = \sum_{i=1}^{n}y_i$.}\\[0.2cm]}
Notice that $h$ can be written as
\[ 
    h(x,y) = \sum_{i=1}^{n}\beta_{i}y_{i}+\lambda^{\ast} \left(s_X - s_Y\right)^{2}-2\lambda^{\ast}\sigma\left(s_X-s_Y\right).
\]
The partial derivatives of $h$ are 
\begin{equation*}
    \frac{\partial h(x,y)}{\partial y_{k}}=\beta_{k} -2\lambda^{\ast}\left(\sum_{i=1}^{n}(x_{i}-y_{i})\right)+2\lambda^{\ast}\sigma, \quad k=1,\ldots,n,
\end{equation*}
and the condition $\beta_i\neq \beta_j$ for $i\neq j$ implies that no $y$ in $\bbR^{n}_{+}$ solves $\partial h(x,y)/\partial y_{k}=0$ for every $k=1,\ldots,n$. Notice that this implies that the minima must occur in the boundary of $\partial A_{R}$. Rewriting $h$ in the form
\begin{align*}
    h(x,y) &= \sum_{i=1}^{n}(\beta_{i}-\beta_{n})y_{i} + \beta_{n}s_Y + \lambda^{\ast}\left(s_Y^{2}+2\left(\sigma-s_X\right)s_Y\right)+\lambda^{\ast}s_X^{2}-2\lambda^{\ast}\sigma s_X\\
    & = -\sum_{i=1}^{n}(\beta_{n}-\beta_{i})y_{i} +\lambda^{\ast}\left(s_Y-s_X+\sigma+\frac{\beta_{n}}{2\lambda^{\ast}}\right)^{2}+\lambda^{\ast}s_X^{2}-2\lambda^{\ast}\sigma s_X\\ 
    &\qquad\quad- \lambda^{\ast}\left(\sigma + \frac{\beta_{n}}{2\lambda^{\ast}}-s_X\right)^{2},
\end{align*}
we see that in the last line only the first two terms depend on $y$, so it suffices to find the $y$'s minimizing the function
\[
a(x,y) = -\sum_{i=1}^{n}(\beta_{n}-\beta_{i})y_{i} +\lambda^{\ast}\left(s_Y-s_X+\sigma+\frac{\beta_{n}}{2\lambda^{\ast}}\right)^{2}.
\]
\textit{Step 3: If $y^{\ast}$ minimizes $a(x,\cdot)$ in $[0,x]$ and $y_{i}^{\ast}>0$ for some $i\geq 2$, then $y_{j}^{\ast} = x_j$ for $j=1,\ldots,i-1$.}\\[0.2cm]
Indeed, arguing by contradiction, assume there exists $j<i$ such that $y_{j}^{\ast} < x_j$. Define $\tilde{y}$ as $\tilde{y}_{k}=y_{k}^{\ast}$ if $k\neq i,j$, $\tilde{y}_{j}=y_{j}^{\ast}+\min(x_{j}-y_{j}^{\ast},y_{i}^{\ast})$ and $\tilde{y}_{i}=y_{i}^{\ast}-\min(x_{j}-y_{j}^{\ast},y_{i}^{\ast})$. Then
\begin{itemize}
    \item either $\tilde{y}_{j}=x_{j}$ or $\tilde{y}_{i}=0$, and
    \item $a(x,y^{\ast})-a(x,\tilde{y}) = (\beta_i-\beta_j)\min(x_{j}-y_{j}^{\ast},y_{i}^{\ast})>0$.
\end{itemize}
The last statement clearly contradicts the minimum property of $y^{\ast}$, so we conclude that the above claim is true. Coincidentally, this also implies that if $y_i^{\ast}=0$ for some $i$, then $y_j^{\ast}=0$ for all $j\geq i$.

From this it follows that for each fixed $x$, there are at most $n$ candidate solutions of the form $(x_1,\ldots,x_{j-1},y_j,0,\ldots,0)$, where $y_j$ is the only point minimizing the mapping
\begin{equation}\label{eq:appendix:ex23:eq1}
    y\mapsto -\sum_{i=1}^{j-1}(\beta_{n}-\beta_{i})x_{i}-(\beta_{n}-\beta_{j})y +\lambda^{\ast}\left(y-\sum_{i=j}^nx_i+\sigma+\frac{\beta_{n}}{2\lambda^{\ast}}\right)^{2}
\end{equation}
in $[0,x_j]$. For each $j=1,\ldots,n$ we denote the associated candidate solution by $y^{(j)}$. Observe that
\begin{equation}\label{eq:appendix:ex23:eq2}
    y^{(j)}_{j} =  \begin{cases}
0, & \text{if } \sum_{i=j}^nx_i<\sigma+\frac{\beta_{j}}{2\lambda^{\ast}} \\
\sum_{i=j}^nx_i- \sigma-\frac{\beta_{j}}{2\lambda^{\ast}}, & \text{if } \sum_{i=j}^nx_i\geq \sigma+\frac{\beta_{j}}{2\lambda^{\ast}},\text{ but } \sum_{i=j+1}^nx_i<\sigma+\frac{\beta_{j}}{2\lambda^{\ast}} \\
x_j, & \text{if } \sum_{i=j+1}^nx_i\geq \sigma+\frac{\beta_{j}}{2\lambda^{\ast}}
\end{cases}.
\end{equation}
Let $j^\ast$ be the index defined by
\[
j^\ast = \max\left\{j\in\{1,\ldots,n\}\mid \sum_{i=j}^nx_i\geq \sigma+\frac{\beta_{j}}{2\lambda^{\ast}}\right\}
\]
or $j^\ast=1$ if the set on the right is empty. We claim that $y^{(j^\ast)}$ minimizes $a(x,\cdot)$. Indeed, notice that for $k<j^{\ast}-1$, we have $y^{(k)}_r=y^{(k+1)}_r=x_r$ for $r=1,\ldots, k$, $y^{(k)}_{k+1}=x_{k+1}-y^{(k+1)}_{k+1}=0$ and $y^{(k)}_r=y^{(k+1)}_r=0$ for $r>k+1$. Hence
\begin{align*}
    a(x,y^{(k)})-a(x,y^{(k+1)}) &= (\beta_n-\beta_{k+1})x_{k+1} + \lambda^{\ast}\left(-\sum_{i=k+1}^nx_i+\sigma+\frac{\beta_{n}}{2\lambda^{\ast}}\right)^{2}\\
    &\qquad-\lambda^{\ast}\left(-\sum_{i=k+2}^nx_i+\sigma+\frac{\beta_{n}}{2\lambda^{\ast}}\right)^{2}\\
    &=\lambda^{\ast}x_{k+1}\left(\sum_{i=k+1}^nx_i-\sigma-\frac{\beta_{k+1}}{2\lambda^{\ast}}\right)\geq 0.
\end{align*}
The only difference between $y^{(j^\ast-1)}$ and $y^{(j^\ast)}$ is in the $j^\ast$-th entry. However, since $y^{(j^\ast)}_{j^\ast}$ minimizes the map in \eqref{eq:appendix:ex23:eq1}, we have $a(x,y^{(j^\ast)})\leq a(x,y^{(j^\ast-1)})$. For $k>j^\ast$,
\[
a(x,y^{(k+1)})-a(x,y^{(k)}) = -\lambda^{\ast}x_{k+1}\left(\sum_{i=k+1}^nx_i-\sigma-\frac{\beta_{k+1}}{2\lambda^{\ast}}\right)\geq 0,
\]
so we only need to show $a(x,y^{(j^\ast)})\leq a(x,y^{(j^\ast+1)})$. However, if $\tilde{y}$ is such that $\tilde{y}_i=x_i$ for $i\leq j^{\ast}$ and zero otherwise, we can easily see that
\[
a(x,y^{(j^\ast+1)})-a(x,\tilde{y}) = -\lambda^{\ast}x_{j^\ast+1}\left(\sum_{i=j^\ast+1}^nx_i-\sigma-\frac{\beta_{j^\ast+1}}{2\lambda^{\ast}}\right)\geq 0,
\]
and the minimal property of $y^{(j^\ast)}_{j^\ast}$ implies $a(x,y^{(j^\ast)})\leq a(x,\tilde{y})$. Hence, $a(x,y^{(j^\ast)})\leq a(x,y^{(j)})$ for every $j=1,\ldots,n$.\\[0.2cm]
{\color{black}\textit{Step 4: The optimal solution is as given in Equation \eqref{ex:ex2-1:expectation_and_variance:solution}.}\\[0.2cm]}
Noticing that \eqref{eq:appendix:ex23:eq2} can be written in the form 
\[
y^{(j^\ast)}_{j^\ast} = \min\left(\left(\sum_{i=j^\ast}^{n}x_i - \frac{\beta_{j^\ast}}{2\lambda^{\ast}}-\sigma\right)_{+},x_{j^\ast}\right),
\]
the condition $\beta_1<\ldots<\beta_n$ implies that the optimal reinsurance contract is as given in \eqref{ex:ex2-1:expectation_and_variance:solution}.

\subsection{Continuation of Example~\ref{ex:ex2-3:VaR_and_expectation}}\label{appb} 
{\color{black}Recall that the functionals are given by
\begin{align*}
    \calP(\eta) &= \int\sum_{i=1}^{n}\beta_{i}y_{i} \; \eta(dx,dy),\\
    \calG(\eta) &= \widehat{\VaR}_{\alpha}({T_{S}}_{\#}\eta) - c,
\end{align*}
where $0<\beta_{1}<\cdots < \beta_{n}$, $c\geq 0$ and $T_{S}:\mathbb{R}^{n}\times \mathbb{R}^{n}\to \bbR$ is the linear operator defined by $T_S(x,y)=\sum_{i=1}^{n}(x_i-y_i)$. Observe that we can restrict ourselves to the case $0< c < \VaR\left(\sum_{i=1}^{n} X_i\right)$ to avoid the optimal contract being the one given by full or no reinsurance.\\[0.2cm]
\textit{Step 1: Both $\calP$ and $\calG$ satisfy the assumptions of Propositions~\ref{prop:support_minimizer}}\\[0.2cm]
Let $\etastar$ be an optimal reinsurance contract and $v^{\ast} = \widehat{\VaR}_{\alpha}({T_{S}}_{\#}\eta)$. Set
\begin{equation}
\calC=\{\eta\in\scrM\mid \widehat{\VaR}_{\alpha}({T_{S}}_{\#}\eta)=v^{\ast}\}.
\end{equation}
Observe that $\etastar\in \calC$, $\calC$ is convex and for every $\eta\in \calC$,
\[
d\calP(\etastar;\eta-\etastar) = \int\sum_{i=1}^{n}\beta_{i}y_{i} \; (\eta-\etastar)(dx,dy) \text{ and } d\calG(\etastar;\eta-\etastar) = 0.
\]
Hence, the conditions of Proposition~\ref{prop:support_minimizer} are satisfied with $p(x,y) = \sum_{i=1}^{n}\beta_{i}y_{i}$ (independent of $\etastar$).\\[0.2cm]
\textit{Step 2: Setting
\begin{align}
\begin{split}\label{eq:set_D_ex}
D_1& = \left\{(x,y)\in \calA_{R}\mid \sum_{i=1}^{n}(x_i-y_i)<v^{\ast}\right\},\\
D_2& = \left\{(x,y)\in \calA_{R}\mid \sum_{i=1}^{n}(x_i-y_i)>v^{\ast}\right\},
\end{split}
\end{align}
then, for every $(x,y)\in\mathrm{Supp}(\etastar)$, we have 
\[
I(x,y)=\{z\in [0,x]\mid (x,z)\in \overline{D_1}\}
\]
if $(x,y)\in D_1$, while $0\in I(x,y)$ if $(x,y)\in D_2$, where $I(x,y)$ is the set over which minimization occurs in \eqref{eq:support_minimum}.}\\[0.2cm]}
Let $(x,y)\in\mathrm{Supp}(\etastar)$ and recall the measures
\[
\mu_{x,y,t,\varepsilon}(A) = \etastar(A)-\etastar(A\cap B_{\varepsilon}(x,y))+\etastar((A-(0,t))\cap B_{\varepsilon}(x,y)),
\]
appearing in Proposition~\ref{prop:support_minimizer}. We wish to identify the sets over which minimization occurs in \eqref{eq:support_minimum}. We proceed as in the proof of that proposition, i.e., we see that for a suitable chosen $t$, we can find a $\delta>0$ such that $\mu_{x,y,t,\varepsilon}\in \calC$ for $\varepsilon<\delta$. Denote the set from the conclusion of Proposition~\ref{prop:support_minimizer} by $I(x,y)$.

Assume first that $(x,y)\in D_1\setminus \partial \calA_{R}$ and let $T$ denote the set of $t$'s such that $(x,y)\in D_1-(0,t)$, i.e., $T=\{-y<t<x-y\mid (x,y+t)\in D_1\}$. For $t\in T$, let
\[
\delta = \min\left\{d((x,y),\partial \calA_{R}),d((x,y),\partial \calA_{R}-(0,t)),d((x,y),\partial D_1),d((x,y),\partial D_1 -(0,t))\right\}.
\]
With this definition, $\delta>0$ and for every $0\leq \varepsilon<\delta$, we have that $\mu_{x,y,t,\varepsilon}\in\scrM$. We claim that, further, $\widehat{\VaR}_{\alpha}({T_{S}}_{\#}\mu_{x,y,t,\varepsilon})=v^{\ast}$. Indeed, notice that $B_{\varepsilon}(x,y)\subset D_1\cap (D_1-(0,t))$ and therefore
\[
{T_{S}}_{\#}\mu_{x,y,t,\varepsilon}((v^{\ast},\infty)) = \mu_{x,y,t,\varepsilon}(D_2) = \etastar(D_2)\leq \alpha.
\]
Now, let $\varepsilon' = \min\left\{d((x,y),\partial D_1),d((x,y),\partial D_1 -(0,t))\right\}-\varepsilon$. We have $\varepsilon'>0$ by choice of $\varepsilon$. For $v^{\ast}-\varepsilon' < u < v^{\ast}$, let 
\[
    D_u = \left\{(x,y)\in \calA_{R}\mid \sum_{i=1}^{n}(x_i-y_i)>u\right\}.
\]
Observe that, for these $u$'s, we still have $B_{\varepsilon}(x,y)\subset D_u^{c}\cap (D_u^{c}-(0,t))$ and therefore
\[
{T_{S}}_{\#}\mu_{x,y,t,\varepsilon}((u,\infty)) = \mu_{x,y,t,\varepsilon}(D_u) = \etastar(D_u)> \alpha.
\]
Hence $\widehat{\VaR}_{\alpha}({T_{S}}_{\#}\mu_{x,y,t,\varepsilon})=v^{\ast}$. In a similar manner, we see that if $t\not\in T$ (but still $-y<t<x-y$), we have $\widehat{\VaR}_{\alpha}({T_{S}}_{\#}\mu_{x,y,t,\varepsilon})>v^{\ast}$ for $\varepsilon$ small enough, so that 
\begin{equation}\label{eq:appendix:exVar_expectation:Ixy}
    I(x,y)=\{z\in [0,x]\mid (x,z)\in \overline{D_1}\}.
\end{equation}

Now we consider $(x,y)\in D_1\cap \partial \calA_{R}$ and proceed similarly by defining
\begin{align*}
    \delta &= \min\left\{d((x,y),\partial \calA_{R}-(0,t)),d((x,y),\partial D_1 -(0,t))\right\},\\
    \varepsilon' &= d((x,y),\partial D_1 -(0,t))-\varepsilon,
\end{align*}
for $t$'s such that $-x_0<t_i\leq 0$ if $y_{0,i}=x_{0,i}$, $0\leq t_i< x_0$ if $y_{0,i}=0$, $-x_{0,i}<t_i< x_{0,i}-y_{0,i}$ if $0<y_{0,i}<x_{0,i}$ and $(x,y)\in D_1-(0,t)$. We still obtain $I(x,y)$ as in \eqref{eq:appendix:exVar_expectation:Ixy}. Thus, for $(x,y)\in D_1$,
\[
I(x,y) = \left\{z\in [0,x]\mid \sum_{i=1}^{n}(x_i-z_i)\leq v^{\ast}\right\}.
\]

Likewise, we can show that for $(x,y)\in D_2$ we can define  $T=\{-y<t<x-y\mid (x,y+t)\in D_2\}$ and obtain
\[
I(x,y) \supset \left\{z\in [0,x]\mid \sum_{i=1}^{n}(x_i-z_i)\geq v^{\ast}\right\}.\footnote{In this case we obtain only an inclusion as opposed to an equality. This comes from the fact that for $t\not\in T$ we can only ensure $\widehat{\VaR}_{\alpha}({T_{S}}_{\#}\mu_{x,y,t,\varepsilon})\leq v^{\ast}$. As a matter of fact, for $\varepsilon$ small enough, we will actually have an equality and therefore $I(x,y)=[0,x]$, though we do not need this information to find the solution.}
\]
In particular, $0\in I(x,y)$ for $(x,y)\in D_2$.\\[0.2cm]
{\color{black}\textit{Step 3: If
\begin{align*}
D_3 = \left\{(x,y)\in \calA_{R}\mid \sum_{i=1}^{n}(x_i-y_i)=v^{\ast}\right\} = \calA_{R}\setminus (D_1\cup D_2)
\end{align*}
then\begin{align*}
    \mathrm{Supp}(\etastar)\subset D_3\cup \{(x,0)\mid x\in \bbR^{n}_{+}\}.
\end{align*}}
\vspace*{0.05cm}}
Let $(x,y)$ be an arbitrary point in the support of $\etastar$. Since $\calA_{R} = D_1\cup D_2\cup D_3$, $(x,y)$ is in one (and only one) of these three sets. If $(x,y)\in D_1$, the computations from the previous step, together with the conclusion of Proposition~\ref{prop:support_minimizer}, show that
\[
y \in \argmin \left\{\sum_{i=1}^{n}\beta_{i}z_{i}\mid z\in [0,x], \sum_{i=1}^{n}x_i-v^{\ast}\leq \sum_{i=1}^{n}z_i\right\}.
\]
This is a linear optimization exercise with linear constraints and one can see that the condition $(x,y)\in D_1$ implies $\sum_{i=1}^{n}x_i\leq v^{\ast}$ and $y=0$. Similarly, if $(x,y)\in D_2$, $\sum_{i=1}^{n}x_i>v^{\ast}$ and $y=0$. Therefore,
\begin{equation}\label{eq:appendix:exVar_expectation:supp_etastar}
    \mathrm{Supp}(\etastar)\subset D_3\cup \{(x,0)\mid x\in \bbR^{n}_{+}\}.
\end{equation}
{\color{black}Before proceeding to the next step, we notice that just as in the proof of Proposition~\ref{prop:suppport_minimum}, it is intuitive that if $(x,y)\in \mathrm{Supp}(\etastar)\cap D_3$, then $y$ should be in the minima of $p(x,\cdot)$. We will prove this later, but first, we are required to identify these minima.\\[0.2cm]
\textit{Step 4: Recalling that $0<\beta_{1}<\cdots < \beta_{n}$, then for fixed $x\in\bbR^{n}$ such that $\sum_{i=1}^{n}x_{j}> v^{\ast}$, the minimum of $y\mapsto\sum_{i=1}^{n}\beta_{i}y_{i}$ on 
\[
[0,x]\cap \left\{y\in \bbR^{n}_+\mid \sum_{i=1}^{n}y_{i}=\sum_{i=1}^{n}x_{i}-v^{\ast}\right\}
\]
is unique and is given by
\begin{equation}\label{eq:appendix:exVar_expectation:ystar}
    y^{\ast}(x) = \left(\min\left(Q_{1}(x),x_1\right),\ldots,\min\left(Q_{n-1}(x),x_{n-1}\right),Q_{n}(x)\right),
\end{equation}
where 
\[
Q_i(x)=\left(\sum_{j=i}^{n}x_{j}-v^{\ast}\right)_{+}, \quad i=1,\ldots,n.
\]
}}
Indeed, let $\tilde{y}$ denote any minimum and for $i=1,\ldots, n$, let $C_i$ be given by
\[
    C_i=\left\{z\in \bbR^{n}_+\mid \sum_{j=i+1}^{n}z_{j}\leq v^{\ast}, \sum_{j=i}^{n}z_{j}> v^{\ast}\right\}.
\]
These sets are disjoint and $x$ belongs to one (and only one) of them. Let $i_0$ denote the index such that $x\in C_{i_0}$. We claim that $\tilde{y}_j=x_j$ for $j<i_0$. For otherwise, we can let $j_0$ be the first index for which this does not happen or $j_0=1$ in case $\tilde{y}_j<x_j$ for all $j$'s. Therefore, $\tilde{y}_j=x_j$ for $j<j_0$, while $\tilde{y}_{j_0}<x_{j_0}$. If $\tilde{y}_k=0$ for all $k>j_0$, then the condition $\sum_{i=1}^{n}\tilde{y}_{i}=\sum_{i=1}^{n}x_{i}-v^{\ast}$ implies
\[
\tilde{y}_{j_{0}}=\sum_{i=j_{0}}^{n}x_{i}-v^{\ast}>x_{j_{0}}
\]
since $j_0<i_0$. Thus, there exists $k_0>j_0$ such that $y_{k_0}>0$. However, in this case, the vector $\hat{y}\in \bbR^{n}$ given by $\hat{y}_{j}=\tilde{y}_j$ if $j\neq j_0,k_0$, $\hat{y}_{j_{0}}=\min(x_{j_0},\tilde{y}_{j_{0}}+\tilde{y}_{k_{0}})$ and $\hat{y}_{k_{0}}=\max(0,\tilde{y}_{k_{0}}+\tilde{y}_{j_{0}}-x_{j_0})$ satisfies all the constraints and further $p(x,\hat{y})<p(x,\tilde{y})$ as $\beta_{j_0}<\beta_{k_0}$. Since this is a contradiction to the optimality of $\tilde{y}$, we therefore conclude $\tilde{y}_{j}=x_j$ for all $j<i_0$ and
\[
\sum_{j=i_{0}}^{n}\tilde{y}_{j}=\sum_{j=i_{0}}^{n}x_{i}-v^{\ast}.
\]
This now implies $\tilde{y}_{i_{0}} = \sum_{j=i_{0}}^{n}x_{i}-v^{\ast}\leq x_{i_{0}}$, since
\begin{align*}
    \beta_{i_{0}}\tilde{y}_{i_{0}} = \beta_{i_{0}}\left(\sum_{j=i_{0}}^{n}x_{i}-v^{\ast}\right)=\beta_{i_{0}}\sum_{j=i_{0}}^{n}y_{j}\leq \sum_{j=i_{0}}^{n}\beta_jy_{j}
\end{align*}
for any other $y\in\bbR^{n}_{+}$ with $y_j=x_j$ for $j<i_0$ (which is a condition for optimality). Hence
\[
\tilde{y} = \left(x_1,x_2,\ldots,x_{i_{0}-1},\sum_{j=i_{0}}^{n}x_{i}-v^{\ast},0,\ldots,0\right),
\]
which is equivalent to \eqref{eq:appendix:exVar_expectation:ystar} on $C_{i_0}$ and shows unicity of $\tilde{y}$.

Note further that for any $x\in \bbR^{n}_+$ with $\sum_{i=1}^{n}x_{j}< v^{\ast}$ we anyway have
\[
y^\ast(x) \in \argmin \left\{\sum_{i=1}^{n}\beta_{i}z_{i}\mid z\in [0,x], \sum_{i=1}^{n}x_i-v^{\ast}\leq \sum_{i=1}^{n}z_i\right\}.
\]
{\color{black}Using this, we now formalize the statement made before Step 4.\\[0.2cm]
\textit{Step 5: We have
\begin{align*}
    \mathrm{Supp}(\etastar)\subset \{(x,0)\mid x\in \bbR^{n}_{+}\}\cup\{\left(x,y^{\ast}(x)\right)\mid x\in \bbR^{n}_{+}\}.
\end{align*}
}}
From \eqref{eq:appendix:exVar_expectation:ystar} it follows that the mapping $x\mapsto y^{\ast}(x)$ is continuous. If for some $(x,y)\in \mathrm{Supp}(\etastar)\cap D_3$, $y\neq 0$, we had $y\neq y^{\ast}(x)$, then, by continuity of $q(x):=p(x,y^{\ast}(x))$, we could find an open ball $B\subset \bbR^{n}_{+}\times \bbR^{n}_{+}$ around $(x,y)$ not intersecting $\bbR^{n}_{+}\times \{0\}$ and such that $q(x') < p(x',y')$ for every $(x',y')\in B$. We could then modify $\etastar$ by moving all the mass in $B$ to the graph of $q$, obtaining a measure with the same VaR and a strictly smaller value for $\calP$. For example, we can use the following measure
\begin{equation*}
    \eta(A) = \etastar(A\setminus B) + \mu(Q^{-1}(A)\cap B').
\end{equation*}
Here $Q=(\mathrm{Id},q)$ and $B'\subset \pi_1(B)$ is an open ball around $x$ such that $\etastar(B)=\mu(B')$, which exists by absolute continuity of $\mu$.
It therefore follows that 
\begin{align*}
    \mathrm{Supp}(\etastar)\subset \{(x,0)\mid x\in \bbR^{n}_{+}\}\cup\{\left(x,y^{\ast}(x)\right)\mid x\in \bbR^{n}_{+}\}.
\end{align*}
\vspace{0.05cm}

\noindent{\color{black}\textit{Step 6: The optimal solution is deterministic and there exists $d\geq 0$ such that if $E=\{x\in \bbR^{n}_{+}\mid q(x)\leq d\}$, the optimal reinsurance contract is given by
\begin{equation}
    R(x) =  \begin{cases}
y^{\ast}(x), & \text{if } x\in E \\
0, & \text{otherwise} .
\end{cases}
\end{equation}
}
From the previous step we only need to show that for $(x,y)\in \mathrm{Supp}(\etastar)$ with $\sum_{i=1}^{n}x_{j}> v^{\ast}$ we either have $y= 0$ or $y = y^\ast(x)$, but not both. Let $d$ be chosen so that $\mu(E)=1-\alpha$, which is possible due to absolute continuity of $\mu$. Note that $d>0$, as otherwise $y^\ast(x)=0$ and the optimal contract is no reinsurance. A similar argument to the previous step shows that, in case $d>0$, there cannot be a point $(x,y)$ in the support of $\etastar$ such that $q(x)>d$. Note that by using the explicit definition of $y^\ast(x)$, we obtain 
\begin{equation*}
    R(x) =  \left(x_1,\ldots,x_{i-1},\sum_{j=i}^{n}x_j-v^{\ast},0,\ldots,0\right)
\end{equation*}
if $\sum_{j=1}^{i-1}\beta_jx_j+\beta_i\sum_{j=i}^{n}x_j-\beta_iv^{\ast}\leq d$, $\sum_{j=i+1}^{n}x_j\leq v^{\ast}$ and $\sum_{j=i}^{n}x_j> v^{\ast}$, and $R(x)=0$ otherwise.}

\subsection{{\color{black}Continuation of Example~\ref{ex:ex2}}}\label{appc}
Recall that we have shown that the optimal contract is deterministic and we are tasked to find $\nu$ such that the mapping \[
x\mapsto g_\nu(x):=F_\nu^{-1}\circ F_{\mu}(x).
\]
is optimal. We can phrase this problem in terms of functions: consider the operator $\calF:L^{2}([0,1])\to \bbR$ given by 
\[
\calF(f) = \int_{0}^{1}\left(F_{\mu}^{-1}(x)-f(x)\right)^{2}\;dx - \left(\int_{0}^{1}F_{\mu}^{-1}(x)-f(x) \; dx\right)^{2}.
\]
We want to minimize this functional subject to the constraints $0\leq f\leq F_{\mu}^{-1}$, $f$ non-decreasing and 
\begin{equation}\label{ex2:ex2_revisited:var_constraint}
\int_{0}^{1}f(x)^{2}\;dx - \left(\int_{0}^{1}f(x) \; dx\right)^{2} = c.
\end{equation}
We can first look at the problem that considers only the last constraint and examine the associated Lagrange operator, $\calL:L^{2}([0,1])\times\bbR \to \bbR$ given by
\[
\calL(f,\lambda) = \calF(f) + \lambda\int_{0}^{1}f(x)^{2}\;dx - \lambda\left(\int_{0}^{1}f(x) \; dx\right)^{2}-\lambda c.
\]
For each $f\in L^{2}([0,1])$ and $\lambda\in \bbR$, the functional derivative of this operator with respect to $h$ is the functional $\calL'(f,\lambda;\cdot):L^{2}([0,1])\to \bbR$ given by
\begin{align*}
\calL'(f,\lambda;h) &= 2\int_{0}^{1}(F_{\mu}^{-1}(x)-f(x))h(x)\;dx-2\int_{0}^{1}(F_{\mu}^{-1}(x)-f(x))\;dx\int_{0}^{1}h(x)\;dx\\
& + 2\lambda\int_{0}^{1}f(x)h(x)\;dx-2\lambda\int_{0}^{1}f(x)\;dx\int_{0}^{1}h(x)\;dx\\
& = 2\int_{0}^{1}\left(F_{\mu}^{-1}(x)+(\lambda-1)f(x)-\int_{0}^{1}(F_{\mu}^{-1}(y)+(\lambda-1)f(y))\;dy\right)h(x)\;dx.
\end{align*}
If $f$ is to be an extreme point of $\calL$, then $\calL'(f,\lambda;h)$ is to be zero for every $h\in L^{2}([0,1])$, which can happen if and only if
\[
F_{\mu}^{-1}(x)+(\lambda-1)f(x)-\int_{0}^{1}(F_{\mu}^{-1}(y)+(\lambda-1)f(y))\;dy = 0, \quad x\in [0,1].
\]
Hence, the function is constant and $F_{\mu}^{-1}=(1-\lambda)f + a$ for some $a\in \bbR$. Regardless of $a$, the function will satisfy \eqref{ex2:ex2_revisited:var_constraint} as long as we choose $\lambda$ such that $(1-\lambda)^{2} = \Var(X)/c$. Observe then that all the other constraints will be satisfied by choosing $\lambda = 1 - \sqrt{\Var(X)/c}$ and any $0 \leq a\leq F_{\mu}^{-1}(0)$.

\subsection{{\color{black}Continuation of Example~\ref{ex:ex6}}}\label{appd}

In this appendix we detail the procedure of binning of the $X_i$'s as well as the setting of the optimization procedure: let $q\in ]0,1[$ and define $u_{i} = F_{X_i}^{-1}(q)$, $i=1,2$. For $N\in\mathbb{N}$, we introduce the variables $\tilde{X}_1$ and $\tilde{X}_2$ such that, for $i=1$ and $i=2$,
\[
\tilde{X}_i = \frac{ku_{i}}{N} \text{ with probability } F_{X_i}\left(\frac{ku_{i}}{N}\right)-F_{X_i}\left(\frac{(k-1)u_{i}}{N}\right), \; k=1,\ldots, N-1
\]
and
\[
\tilde{X}_i = u_i \text{ with probability } 1-F_{X_i}\left(\frac{(N-1)u_{i}}{N}\right).
\]
Let $\tilde{\mu}$ denote the joint distribution of $(\tilde{X}_1,\tilde{X}_2)$ obtained from the joint distribution of $(X_1,X_2)$ as the push-forward of the binning function. As in the main part of the text, we let $\tilde{Y}_1$ and $\tilde{Y}_2$ be two random variables such that
\[
\tilde{Y}_1 \overset{d}{=} 0.5\tilde{X}_1, \quad \tilde{Y}_2 \overset{d}{=} \min(\tilde{X}_2,0.5)+0.25(\tilde{X}_2-0.95)_+,
\]
where $\overset{d}{=}$ denotes equality in distribution.Let $\tilde{\nu}_1$ and $\tilde{\nu}_2$ denote the distributions of $\tilde{Y}_1$ and $\tilde{Y}_2$. Our objective is setting a discrete (multi-marginal) OT problem with source distribution $\tilde{\mu}$ and targets $\tilde{\nu}_1$ and $\tilde{\nu}_2$ that minimizes the variance of $(\tilde{X}_1-\tilde{Y}_1)+(\tilde{X}_2-\tilde{Y}_2)$ (the reinsured amount).\\
By letting $M$ denote the amount of distinct values taken by $\tilde{Y}_2$, in the following we will implicitly assume that these random variables and distributions are defined on the space $\Omega' = \{1,\ldots, N\}^3\times \{1,\ldots, M\}$ and are such that, for example, $\tilde{Y}_1(i,j,k,l)=\tilde{Y}_1(k)$ is the $k$-th value taken by $\tilde{Y}_1$ in increasing order and $\tilde{\nu}_1(k) = \bbP[\tilde{Y}_1=\tilde{Y}_1(k)]$. Similar definitions apply to the other random variables.
The optimal transport problem is then equivalent to finding a 4-dimensional array $P$ in $\bbR^{N\times N\times N\times M}_{+}$ such that
\begin{equation}\label{eq:linear_program}
    \begin{split}
    \sum_{k,l} P_{i,j,k,l} = \tilde{\mu}(i,j),& \quad (i,j)\in \{1,\ldots, N\}^2,\\
    \sum_{i,j,l} P_{i,j,k,l} = \tilde{\nu}_1(k&), \quad k\in \{1,\ldots, N\},\\
    \sum_{i,j,k} P_{i,j,k,l} = \tilde{\nu}_2(l&), \quad l\in \{1,\ldots, N\},\\
    P_{i,j,k,l} = 0 \text{  if  } \tilde{Y}_1(k)>&\tilde{X}_1(i) \text{  or  } \tilde{Y}_2(l)>\tilde{X}_2(j),
    \end{split}
\end{equation}
and $P$ minimizes the sum 
$$\sum_{i,j,k,l}(\tilde{X}_1(i)-\tilde{Y}_1(k)+\tilde{X}_2(j)-\tilde{Y}_2(l))^2P_{i,j,k,l}$$
among all arrays satisfying \eqref{eq:linear_program}. This is a linear optimization problem. We want to use standard linear optimization techniques to solve it, so we cast these equations into standard form: we start by ``flattening'' $P$ into an element $p\in \bbR^{N^3M}$ by making the $(i,j,k,l)$ entry of $P$ into the $i+N(j-1)+N^2(k-1)+N^3(l-1)$ entry of $p$. Similarly, we let the cost be represented by a vector $c\in \bbR^{N^3M}$ with $i+N(j-1)+N^2(k-1)+N^3(l-1)$ entry equal to $(\tilde{X}_1(i)-\tilde{Y}_1(k)+\tilde{X}_2(j)-\tilde{Y}_2(l))^2$. Denoting by $\mathbbm{1}_N$ the column vector of dimension $N$ filled with ones, $\bbI_N$ the identity matrix of dimension $N$ and $\otimes$ the Kronecker product, the $(N^2+N+M)\times N^3M$ matrix
\[
A = \begin{bmatrix}
\mathbbm{1}_{NM}^\intercal\otimes \bbI_{N^2}\\
\mathbbm{1}_{M}^\intercal\otimes \bbI_{N}\otimes \mathbbm{1}_{N^2}^\intercal\\
\bbI_{M}\otimes \mathbbm{1}_{N^3}^\intercal
\end{bmatrix}
\]
can be used to encode the first three constraints in \eqref{eq:linear_program}. Indeed, if $\theta\in \bbR^{N^2+N+M}$ is given as the stacking of $\tilde{\mu}$, $\tilde{\nu}_1$ and $\tilde{\nu}_2$ (flattening first $\tilde{\mu}$, so that $\tilde{\mu}(i,j)$ is the $i+N(j-1)$ entry of $\theta$), we see that $P$ satisfies the first three equalities in \eqref{eq:linear_program} if and only if $Ap = \theta$. Finally, we can enforce the last constraint in \eqref{eq:linear_program} by deleting the columns and entries of $A, p$ and $c$ for which we have $\tilde{Y}_1(k)>\tilde{X}_1(i)$ or $\tilde{Y}_2(l)>\tilde{X}_2(j)$ (that is, we delete the $i+N(j-1)+N^2(k-1)+N^3(l-1)$ column of $A$ if this condition is satisfied). Let $B, q$ and $d$ denote, respectively, the matrix and vectors obtained this way and $K$ the number of columns that remained after this operation. The optimization problem then is of the form
\[
\text{Minimize } d^\intercal q \text{ subject to } Bq = \theta \text{ and } q\geq 0.
\]
Although potentially high-dimensional, this is a relatively easy linear optimization exercise. We solve the problem for $N=40$, using the \texttt{linprog} routine within the \texttt{SciPy} library, which implements the HiGHS software for linear optimization.

\bibliographystyle{abbrv}
\bibliography{main}
\end{document}